\documentclass[11pt,a4paper]{amsart}
\usepackage{amssymb,xspace} 
\usepackage{amstext} 
\usepackage[mathscr]{eucal} 
\theoremstyle{plain}
\usepackage{amsbsy,amssymb,amsfonts,latexsym,eucal,amscd}
\usepackage[dvips]{graphicx}
\usepackage{epsfig}
\usepackage[all]{xy}
\usepackage{pstricks}
\newcommand{\ds }{\ensuremath{\displaystyle}}

\newcommand{\R }{\ensuremath{\mathbb R}}
\newcommand{\C }{\ensuremath{\mathbb C}}
\newcommand{\Q }{\ensuremath{\mathbb Q}}
\newcommand{\Z }{\ensuremath{\mathbb Z}}
\newcommand{\N }{\ensuremath{\mathbb N}}

\newcommand{\T }{\ensuremath{\mathbb T}}
\renewcommand{\P }{\ensuremath{\mathbb P}}
\newcommand{\HHH }{\ensuremath{\mathbb H}}
\newcommand{\hoo}{\ensuremath{\mathcal{H}om}}

\DeclareMathOperator{\rel}{\vphantom{y}rel}

\DeclareMathOperator{\ch}{ch}

\DeclareMathOperator{\supp}{supp}
\DeclareMathOperator{\td}{td}

\newcommand{\tor}{\ttt\! or}
\newcommand{\ext}{\eee xt}

\DeclareMathOperator{\pr}{pr}

\DeclareMathOperator{\gr}{Gr}
\DeclareMathOperator{\coh}{Coh}

\newcommand{\aaa }{\ensuremath{\mathcal{A}}}
\newcommand{\bbb }{\ensuremath{\mathcal{B}}}

\newcommand{\oo }{\ensuremath{\mathcal{O}}}
\newcommand{\hh }{\ensuremath{\mathcal{H}}}
\newcommand{\ff }{\ensuremath{\mathcal{F}}}

\newcommand{\g }{\ensuremath{\mathcal{G}}}
\newcommand{\ttt }{\ensuremath{\mathcal{T}}}
\newcommand{\jj }{\ensuremath{\mathcal{J}}}
\newcommand{\mm }{\ensuremath{\mathcal{M}}}
\newcommand{\nn }{\ensuremath{\mathcal{N}}}
\newcommand{\eee }{\ensuremath{\mathcal{E}}}

\newcommand{\LL }{\ensuremath{\mathcal{L}}}

\newtheorem{theorem}{Theorem}[section]

\newtheorem{lemma}[theorem]{Lemma}

\newtheorem{proposition}[theorem]{Proposition}

\newtheorem{corollary}[theorem]{Corollary}

{\theoremstyle{definition}}

{\theoremstyle{definition}}

{\theoremstyle{definition}}

{\theoremstyle{definition}}

{\theoremstyle{definition}\newtheorem{definition}[theorem]{Definition}}

{\theoremstyle{definition}}

{\theoremstyle{definition}\newtheorem{remark}[theorem]{Remark}}

\address{Institut de Math\'{e}matiques de Jussieu, UMR 7586\\
Case 247\\ Universit\'{e} Pierre et Marie Curie\\
4, place Jussieu\\
F-75252 Paris Cedex 05\\
France}
\email{jgrivaux@math.jussieu.fr}

\newcommand{\tn}{\T^{[n]}\be}

\newcommand{\hf}{H\raisebox{-1.3ex}{$\mkern-20mu\leftarrow$}\he{}}

\newcommand{\oti}{\ensuremath{\otimes}}

\newcommand{\he}{^{\vphantom{[n]}} }
\newcommand{\be}{_{\vphantom{[n]}} }

\newcommand{\ba}[1]{\ensuremath{\overline{#1}}}

\newcommand{\ti }[1]{\ensuremath{\widetilde{#1}}}

\newcommand{\rb }{\ensuremath{\raisebox}}

\newcommand{\tim }{\ensuremath{\times}}

\newcommand{\ee }{\ensuremath{^{\, *}}}

\newcommand{\yn}{Y_{n}\he}

\newcommand{\tw}{\ti{W}}

\newcommand{\bop }{\ensuremath{\bigoplus\limits}}

\newcommand{\suq }{\ensuremath{\subseteq}}

\newcommand{\pee }{\ensuremath{\pe }}

\newcommand{\pe }{\ensuremath{^{\, !}}}

\entrymodifiers={+!!<0pt,\fontdimen22\textfont2>}

\def\apl#1#2#3{#1\mkern -2 mu:\mkern - 4 mu
\xymatrix@C=17pt{#2\!\ar[r]&\!#3}
}

\def\apliso#1#2#3{#1\mkern -2 mu:\mkern - 4 mu
\xymatrix@C=17pt{#2\!\ar[r]^-{\sim}&\!#3}
}

\def\aplpt#1#2#3#4{#1\mkern -4 mu:\mkern - 8 mu
\xymatrix@C=17pt{#2\!\ar[r]&\!#3#4}
}

\def\sutrgd#1#2#3{
\xymatrix@C=17pt{
0\ar[r]&#1\ar[r]&#2\ar[r]&#3\ar[r]&0
}
}

\def\sutrgdpt#1#2#3#4{
\xymatrix@C=17pt{
0\ar[r]&#1\ar[r]&#2\ar[r]&#3\ar[r]&0#4
}
}

\def\sutrgpt#1#2#3#4{
\xymatrix@C=17pt{
0\ar[r]&#1\ar[r]&#2\ar[r]&#3#4
}
}

\def\sutr#1#2#3{
\xymatrix@C=17pt{
#1\ar[r]&#2\ar[r]&#3
}
}

\def\sutrg#1#2#3{
\xymatrix@C=17pt{
0\ar[r]&#1\ar[r]&#2\ar[r]&#3
}
}

\def\sutrd#1#2#3{
\xymatrix@C=17pt{
#1\ar[r]&#2\ar[r]&#3\ar[r]&0
}
}

\def\fl{\xymatrix@C=17pt{
\ar[r]&
}}

\def\flcourte{\xymatrix@C=10pt{
\ar[r]&
}}

\def\flgd#1#2{\xymatrix@C=17pt{#1\!
\ar[r]&\!#2
}}

\def\flgdba#1#2{\xymatrix@C=15pt{#1\!
\ar@{|->}[r]&#2
}}

\def\flcourtegd#1#2{\xymatrix@C=15pt{\!\!#1\!
\ar[r]&\!#2\!\!
}}

\def\flgdbain#1#2{\xymatrix@C=3ex{\!\!\scriptstyle{#1}
\ar[r]&\scriptstyle{#2}\!\!\!
}}

\def\fldouble{\xymatrix@1{
\ar@{->>}[r]&
}}

\def\fle#1#2{
\xymatrix@1{
#1
\ar[r]&#2
}}

\def\flex#1#2#3{
{\xymatrix@1{
#1
\ar[r]^{#3}&#2
}}
}

\def\fledouble#1#2{
{\xymatrix@1{
#1
\ar@{->>}[r]&{#2}
}}
}

\def\flexdouble#1#2#3{
{\xymatrix@1{
#1
\ar@{->>}[r]^{#3}&{#2}
}}
}

\def\diagca#1#2#3#4#5#6#7#8{\xymatrix@1{
#1
\ar[d]_{#6}\ar[r]_{#5}&#2\ar[d]_{#7}\\
#3
\ar[r]_{#8}&#4
}}

\def\sutrois#1#2#3{
{\xymatrix@1{
#1
\ar[r]&#2
\ar[r]&#3
}}
}

\def\sutroiszerogdprime#1#2#3{
{\xymatrix@1{
0
\ar@<-0.5mm>[r]&#1
\ar@<-0.5mm>[r]&#2
\ar@<-0.5mm>[r]&#3
\ar@<-0.5mm>[r]&0
}}
}

\def\fleprime#1#2{
\xymatrix@1{
#1
\ar[r]&#2
}}

\def\sutroisnom#1#2#3#4#5{
{\xymatrix@1{
#1
\ar[r]^{#4}&#2
\ar[r]^{#5}&#3
}}
}

\def\sutroiszerogd#1#2#3{
{\xymatrix@1{
0
\ar[r]&#1
\ar[r]&#2
\ar[r]&#3
\ar[r]&0
}}
}
\def\strgdexp#1#2#3#4#5#6{
{\xymatrix@1{
0
\ar[r]&\rb{#2ex}{$#1$}
\ar[r]&\rb{#4ex}{$#3$}
\ar[r]&\rb{#6ex}{$#5$}
\ar[r]&0
}}
}

\def\sutroiszerog#1#2#3{
{\xymatrix@1{
0
\ar[r]&#1
\ar[r]&#2
\ar[r]&#3
}}
}

\def\suxtroiszerogd#1#2#3#4#5{
{\xymatrix@1{
0
\ar[r]&#1
\ar[r]^{#4}&#2
\ar[r]^{#5}&#3
\ar[r]&0
}}
}

\def\suquatre#1#2#3#4{
{\xymatrix@1{
#1
\ar[r]&#2
\ar[r]&#3
\ar[r]&#4
}}
}

\def\suxquatre#1#2#3#4#5#6#7{
{\xymatrix@1{
#1
\ar[r]^{#5}&#2
\ar[r]^{#6}&#3
\ar[r]^{#7}&#4
}}
}

\def\sucinq#1#2#3#4#5{
{\xymatrix@1{
#1
\ar[r]&#2S^{n}X
\ar[r]&#3
\ar[r]&#4
\ar[r]&#5
}}S^{n}X
}

\def\suxcinq#1#2#3#4#5#6#7#8#9{
{\xymatrix@1{
#1
\ar[r]^{#6}&#2
\ar[r]^{#7}&#3
\ar[r]^{#8}&#4
\ar[r]^{#9}&#5
}}
}

\DeclareMathOperator{\id}{id}

\newcommand{\yg}{\ensuremath{\mathfrak{Y} }}
\newcommand{\ygn}{\ensuremath{\mathfrak{Y}_{n} }}
\newcommand{\xg}{\ensuremath{\mathfrak{X} }}
\newcommand{\xgp}{\ensuremath{\mathfrak{X}' }}
\newcommand{\zg}{\ensuremath{\mathfrak{Z} }}
\newcommand{\zgp}{\ensuremath{\mathfrak{Z}' }}
\newcommand{\sg}{\ensuremath{\mathfrak{S}}}
\newcommand{\re}{\ensuremath{\, \rel}}
\newcommand{\ci}{\ensuremath{\mathcal{C}^{\mkern 1 mu\infty\vphantom{_{p}}}}}
\newcommand{\cio}{\ensuremath{\mathcal{C}^{\mkern 2 mu\omega \vphantom{_{p}}}}}

\newcommand{\pru}{\ensuremath{\pr_{1}^{-1}}}

\newcommand{\snx}{S^{\mkern 1 mu n\!}_{\vphantom{[}} X}

\newcommand{\xn}{X^{[n]}\be}
\newcommand{\hb}{H\ee\be}

\newcommand{\xb}{\ub{x}}
\newcommand{\yb}{\ub{y}}

\newcommand{\ub}[1]{\underline{\vphantom{!}\vphantom{y}#1}}
\newcommand{\jre}{J^{\re }\be}
\newcommand{\jrel}[1]{J^{\re }_{#1}}
\newcommand{\jjrel}[1]{\jj^{\re }_{#1}}

\newcommand{\nbh}{neighbourhood}

\newcommand{\xna}[1]{X^{[#1]}\be}

\newcommand{\kre}[1]{K^{\re}_{\vphantom{f}#1}}
\newcommand{\ore}[1]{\oo^{\re}_{#1}}
\newcommand{\rsas}{relative smooth analytic space}
\newcommand{\rc}{relatively coherent}
\newcommand{\sh}{sheaf}
\newcommand{\sv}{sheaves}
\newcommand{\baf}{\ba{\ff}}
\newcommand{\bag}{\ba{\g}}
\newcommand{\oz}{\oo_{Z}\he}
\newcommand{\tens}[1][]{\mathbin{\otimes_{\raise1.5ex\hbox to-.1em{}{#1}}}}

\title[Topological properties of punctual Hilbert schemes]{Topological properties of punctual Hilbert schemes of almost-complex fourfolds (II)}
\author{Julien Grivaux}
\begin{document}
\begin{abstract}
 In this article, we study the rational cohomology rings of Voisin's punctual Hilbert schemes $\xn$ associated to a symplectic compact fourfold $X$. We prove that these rings can be universally constructed from $H\ee\be(X,\Q)$ and $c_{1}\he(X)$, and that Ruan's crepant resolution conjecture holds if $c_{1}\he(X)$ is a torsion class. Next, we prove that for any almost-complex compact fourfold $X$, the complex cobordism class of $\xn$ depends only on the cobordism class of $X$. 
\end{abstract}
\maketitle
\tableofcontents
\thispagestyle{empty}
\newpage
\setcounter{section}{0}
\section{Introduction}
The punctual Hilbert schemes of a smooth quasi-projective surface have been intensively studied in the past twenty years, starting with the works of G\"{o}ttsche, Nakajima, Grojnowski, Lehn and others (see e.g. \cite{SchHilGo2}). These Hilbert schemes present a strong geometric interest because, among many other properties, they are smooth crepant resolutions of the symmetric products of the surfaces. If $X$ is a $K3$ surface, $\xn$ are hyperk\"{a}hler varieties by a result of Beauville \cite{SchHilBe} and Ruan's conjecture predicts that for all positive integer $n$, 
$H\ee\be\bigl( \xn,\C\bigr)$ is isomorphic as a ring to the orbifold cohomology ring of $\snx$ defined by Chen and Ruan \cite{SchHilCR}, \cite{SchHilALR}. This is obtained by putting together results of Lehn-Sorger on the Hilbert scheme side \cite{SchHilLeSoBis} with the computations done in \cite{SchHilFG} and \cite{SchHilUr} on the orbifold side. The isomorphisms between $H\ee\be\bigl( \xn,\C\bigr)$ and $H\ee_{\textrm{CR}}\bigl( \snx,\C\bigr)$ made it clear that the cohomology rings of punctual Hilbert schemes of a $K3$ surface depend only on the deformation class of the complex structure on $X$ in the space of almost-complex ones, since Chen-Ruan cohomology is a purely almost-complex theory. This is more generally the case for arbitrary projective surfaces: the description of the cohomology ring of Hilbert schemes done in \cite{SchHilLe} and \cite{SchHilLQWMA} shows that $H\ee\be\bigl( \xn,\Q\bigr)$ depends only on the ring $H\ee\be(X,\Q)$ and on the first Chern class of $X$ in $H^{2}\be(X,\Q)$. In the same spirit, it is proved in \cite{SchHilEGL} that the complex cobordism class of $\xn$ depends only on the complex cobordism class of $X$. These facts received an explanation by a recent construction of Voisin \cite{SchHilVo1}: for any almost-complex compact fourfold $X$ and any positive integer $n$, it is possible to construct a punctual Hilbert scheme $\xn$ which is a stable almost-complex differentiable manifold of real dimension $4n$. Besides, $\xn$ is symplectic if $X$ is symplectic \cite{SchHilVo2}. 
\par\medskip 
In Nakajima's theory \cite{SchHilNa}, the correspondence action of the incidence schemes on the cohomology groups of Hilbert schemes is the main ingredient to understand their additive structure via representation theory. This approach has been carried out for the almost-complex case in \cite{SchHilGri}. Our first aim in this paper is now to study the multiplicative structure of the rational cohomology rings of almost-complex Hilbert schemes, generalizing the work of Lehn \cite{SchHilLe} and Li-Qin-Wang \cite{SchHilLQWMA}. We obtain a satisfactory result in the symplectic case:
\begin{theorem}\label{Ch4suite2ThIntroUn}
 If $(X,\omega )$ is a symplectic compact four-manifold, the rings $H\ee\be\bigl( \xn,\Q\bigr)$ can be constructed by universal formulae from the ring $H\ee\be(X,\Q)$ and the first Chern class of $X$ in $H^{2}\be(X,\Q)$.
\end{theorem}
When $X$ is projective, Lehn uses in an essential way algebraic curves on $X$ via their symmetric products imbedded in the Hilbert schemes to determine some excess intersection terms arising in correspondences computations. This argument fails a priori as soon as $X$ is not algebraic, since there is no pseudo-holomorphic curve on $X$ for a generic almost-complex structure on it. However, Donaldson's theorem on symplectic submanifolds \cite{SchHilDo} allows to produce many paseudo-holomorphic curves for arbitrary small perturbations of a fixed almost-complex structure. As a corollary of an effective version of Theorem \ref{Ch4suite2ThIntroUn}, we obtain the cohomological crepant resolution conjecture for symplectic fourfolds with torsion first  Chern class:
\begin{theorem}\label{Ch4suite2ThIntroDeux}
 Let $(X,\omega )$ be a symplectic four-manifold with zero first Chern class in $H^{2}\be(X,\Q)$. Then for any positive integer $n$, $H\ee\be\bigl( \xn,\C\bigr)$ is isomorphic as a ring to $H\ee_{\textrm{CR}}\bigl( \snx,\C\bigr)$. 
\end{theorem}
This result provides many examples where Ruan's conjecture holds in the symplectic category.
\par\medskip 
The second part of this article deals with the cobordism classes of almost-complex Hilbert schemes. The main result is the following:
\begin{theorem}\label{Ch4suite2ThIntroTrois} 
Let $(X,J)$ be an almost-complex compact four-manifold. For any positive integer $n$, the complex cobordism class of $\xn$ given by its stable almost-complex structure depends only on the complex cobordism class of $X$. 
\end{theorem}
The proof of \cite{SchHilEGL} for projective surfaces is rather delicate to adapt here because coherent sheaves, contrary to algebraic cycles or vector bundles, have no equivalent in almost-complex geometry. However, it turns out that the classical proof can be carried out in a relative context on spaces fibered in smooth analytic sets over a singular differentiable basis.
\par\medskip 
 \noindent\textbf{Acknowledgement.} I want to thank my advisor Claire Voisin whose work on the almost-complex punctual Hilbert scheme is at the origin of this article. Her deep knowledge of the subject and her numerous advices have been most valuable to me. I also wish to thank her for her kindness and her patience.
\section{The almost-complex Hilbert scheme}
In this section, we recall definitions and fundamental properties of almost-complex Hilbert schemes for the reader's convenience. More details can be found in \cite{SchHilGri}, \cite{SchHilVo1} and \cite{SchHilVo2}.
\subsection{Definitions and basic properties}
Let $(X,J)$ be an almost-complex compact fourfold. For all $n\in\N\ee\be$, we introduce the \emph{incidence set} 
\[
Z_{n}\he=\bigl\{(\xb,p)\ \textrm{in}\ \snx\tim X\ \textrm{such that}\ p\in X\bigr\}\cdot
\]
Let us fix a Riemannian metric $g$ on $X$, and $\varepsilon $ sufficiently small.
\begin{definition}\label{Ch4suite2DefUn}
Let $\bbb$ be the set of pairs $(W,\jre)$ such that 
\begin{enumerate}
 \item [(i)] $W$ is a \nbh\ of $Z_{n}\he$ in $\snx\tim X$.
\item[(ii)] $\jre$ is a relative integrable structure on the fibers of $\apl{\pr_{1}\he}{W}{\snx}$ depending smoothly  on the parameter $\xb$ in $\snx$.
\item[(iii)] $\sup_{{\xb\in\snx},\ {p\in W_{x}\he}}\he\bigl|\bigl|\jrel{\xb}(p)-J_{\xb}\he(p)\bigr|\bigr|_{g}\he<\varepsilon $.
\end{enumerate}
 For $\varepsilon $ small enough, $\bbb_{\varepsilon }\he$ is nonempty and weakly contractible.
\end{definition}
\begin{definition}\label{Ch4suite2DefDeux}
 Let $(W,\jre)$ be a relative integrable structure in $\bbb$. The \emph{topological Hilbert scheme} $X^{[n]}_{\jre}$ is the subset of $W^{[n]}_{\re}$ defined by 
$
X^{[n]}_{\jre}=\bigl\{(\xb,\xi)\ \textrm{in}\ W^{[n]}_{\re}\ \textrm{such that}\ HC(\xi )=\xb\, \bigr\}
$, where $\apl{HC}{W^{[n]}_{\xb}}{S^{n}\be W_{\xb}\he}$ is the usual Hilbert-Chow morphism.
\end{definition}
To obtain a differentiable structure on $X^{[n]}_{\jre}$, Voisin uses relative integrable structures in a contractible subset $\bbb'$ of $\bbb$ satisfying some additional geometric conditions. Her main result is the following:
\begin{theorem}\cite{SchHilVo1}, \cite{SchHilVo2}\label{Ch4suite2ThUn}
Let $\jre$ be a relative integrable structure in $\bbb'$. Then 
\begin{enumerate}
 \item [(i)] $X^{[n]}_{\jre}$ has a natural differentiable structure. Furthermore, if $J'{\vphantom{J}}^{\re}\be$ is another relative integrable structure in $\bbb'$, $X^{[n]}_{\jre}$ and $X^{[n]}_{J'\be{\vphantom{J}}^{\re}\be}$ are canonically isomorphic modulo diffeomorphisms isotopic to the identity.
\item[(ii)] There is a canonical Hilbert-Chow map $\apl{HC}{X^{[n]}_{\jre}}{\snx}$ whose fibers are homeomorphic to the fibers of the usual Hilbert-Chow morphism for projective surfaces.
\item[(iii)] $X^{[n]}_{\jre}$ can be endowed with a stable almost-complex structure. 
\item[(iv)] If $X$ is symplectic and $\jre$ is compatible with the symplectic structure,  $X^{[n]}_{\jre}$ is also symplectic.
\end{enumerate}
 
\end{theorem}
For arbitrary relative integrable structures, this theorem has the following topological form:
\begin{theorem}\cite{SchHilGri}\label{Ch4suite2ThDeux}
 $\he$\par
\begin{enumerate}
 \item [(i)] Let $\jre$ be a relative integrable structure in $\bbb$. Then $X^{[n]}_{\jre}$ is a topological manifold.
\item[(ii)] If $\bigl\{\jrel{t}\bigr\}_{t\in B(0,r)\suq\R^{d}\be}\he$ is a smooth path in $\bbb$, then the associated relative topological Hilbert scheme $\bigl( \xn,\bigl\{\jrel{t}\bigr\}_{t\in B(0,r)}\he\bigr)$ over $B(0,r)$ is a topological fibration.
\item[(iii)] For any $\xb$ in $\snx$ and any integrable structure in a \nbh\ $U_{\xb}\he$ of $\xb$ sufficiently close to $J$, the Hilbert-Chow morphism $\apl{HC}{\xn}{\snx}$ is locally isomorphic over a \nbh\ of $\xb$ to the classical Hilbert-Chow morphism $\apl{HC}{U^{[n]}_{\xb}}{S^{n}\be U_{\xb}\he}$.
\end{enumerate}

\end{theorem}
Point (ii) implies that for any couple $\bigl( \jre,J'{\vphantom{J}}^{\re}\be\bigr)$ of relative integrable structures in $\bbb$, $X^{[n]}_{\jre}$ and $X^{[n]}_{J'{\vphantom{J}}^{\re}\be}$ are homeomorphic and the isotopy class of this homeomorphism is canonical. Therefore, there exist canonical rings $H\ee\be\bigl( \xn,\Q\bigr)$ and $K\bigl( \xn\bigr)$ such that for every relative  integrable structure $\jre$ in $\bbb$,  $H\ee\be\bigl( \xn,\Q\bigr)$ is canonically isomorphic to $H\ee\be\bigl( X^{[n]}_{\jre},\Q\bigr)$ and $K\bigl( \xn\bigr)$ to $ K\bigl( X^{[n]}_{\jre}\bigr)$.
\par
Point (iii) implies that G\"{o}ttsche's classical formula for the Betti numbers of punctual Hilbert schemes also holds in the almost-complex case (see \cite{SchHilGriTh}). 
\subsection{Incidence varieties and Nakajima operators}
If $n$, $m$ $\in\N\ee\be$, let 
\[
Z_{n\tim m}\he=\bigl\{(\xb,\,\yb,\,p)\ \textrm{in}\ \snx\tim\snx\tim X\ \textrm{such that}\ p\in\xb\cup\yb\bigr\}\cdot
\]
Relative integrable structures in a \nbh\ of $Z_{n\tim m}\he$ will be denoted by $\jrel{n\tim m}$.
\begin{definition}\label{Ch4suite2DefTrois}
$\he$\par
\begin{enumerate}
 \item [(i)] If $\bigl(W_{1}\he, J^{1,\re}_{n\tim m}\bigr)$ and $\bigl(W_{2}\he, J^{2,\re}_{n\tim m}\bigr)$ are two relative integrable structures in \nbh s of $Z_{n\tim m}\he$, the \emph{product Hilbert scheme} $\bigl( X^{[n]\tim[m]}\be,J^{1,\re}_{n\tim m},J^{2,\re}_{n\tim m}\bigr)$ is defined by 
\begin{align*}
 \bigl( X^{[n]\tim[m]}\be,J^{1,\re}_{n\tim m},J^{2,\re}_{n\tim m}\bigr)=\bigl\{
(\xi ,\,\xi ',\,\xb,\,\yb)\ \textrm{in}\  W^{[n]}_{1,\re}&\tim_{\snx\tim S^{m}\be X}\he W^{[m]}_{2,\re}\\& \textrm{such that}\ 
HC(\xi )=\xb\ \textrm{and}\ HC(\xi ')=\yb
\bigr\}\cdot
\end{align*}
\item[(ii)] If $m>n$ and $\bigl( \ti{W},\jrel{n\tim(m-n)}\bigr)$ is a relative integrable structure in a \nbh\ of $Z_{n\tim(m-n)}$, the \emph{incidence variety} is defined by 
\begin{align*}
 \bigl(\xna{m,n},\jrel{n\tim m}\bigr)=\bigl\{
(\xi ,\,\xi ',\,&\xb,\,\yb)\ \textrm{in}\ \ti{W}^{[n]}_{\re}\tim_{\snx\tim S^{m}\be X}\he \ti{W}^{[m]}_{\re}\\ &\textrm{such that}\ \xi \subset\xi ',\ HC(\xi )=\xb\ \textrm{and}\ HC(\xi ')=\xb\cup\yb
\bigr\}\cdot
\end{align*}
\end{enumerate}
\end{definition}
As it is the case for topological Hilbert schemes, the product Hilbert scheme and the incidence varieties are uniquely defined up to homeomorphisms isotopic to the identity. 
\par\medskip 
Let $\bigl( W,\jrel{\vphantom{(}n\tim m}\bigr)$, $\bigl( \ti{W},\jrel{n\tim(m-n)}\bigr)$, $\bigl( W',\jrel{\vphantom{(}n}\bigr)$ and $\bigl( W'',\jrel{\vphantom{(}m}\bigr)$ be relative integrable structures in \nbh s of $Z_{\vphantom{(}n\tim m}$, $Z_{n\tim(m-n)}$, $Z_{\vphantom{(}n}$ and $Z_{\vphantom{(}m}$. We introduce the following set of compatibility conditions of relative analytic spaces (see Section \ref{n}):
\begin{enumerate}
 \item [(1)] $W\tim_{\snx\tim S^{m}\be X}\he(\snx\tim S^{m-n}\be X)\suq\ti{W}$, where the base change map is $\flgdba{(\xb,\yb)}{(\xb,\xb\cup\yb).}$
\item[(2)] $W'\tim_{\snx}\he(\snx\tim S^{m-n}\be X)\suq\ti{W}$, where the base change map is the first projection.
\item[(3)] $W''\tim_{S^{m}\be X}\he(\snx\tim S^{m-n}\be X)\suq\ti{W}$, where the base change map is $\flgdba{(\xb,\yb)}{\xb\cup\yb.}$
\end{enumerate}
-- If (1) holds, there is a natural embedding of $\bigl( \xna{m,n},\jrel{n\tim(m-n)}\bigr)$ into $\bigl( X^{[n]\tim[m]}\be,\jrel{n\tim m},\jrel{n\tim m}\bigr).$
\par 
\noindent-- If (2) holds, there is a natural morphism $\lambda $ from $\bigl( \xna{m,n},\jrel{n\tim(m-n)}\bigr)$ to $\bigl( \xn,\jrel{\vphantom{(}n}\bigr)$.
\par 
\noindent-- If (3) holds, there is a canonical morphism $\nu $ from $\bigl( \xna{m,n},\jrel{n\tim(m-n)}\bigr)$ to $\bigl( X^{[m]}\be,\jrel{\vphantom{(}m}\bigr)$. 
\par\medskip 
\noindent Unfortumately, conditions (2) and (3) cannot hold at the same time. However, since $X^{[n]\tim[m]}\be$ is canonically homeomorphic up to isotopy to $\xn\tim X^{[m]}\be$, one still has morphisms from $\xna{m,n}$ to $\xn$ and $X^{[m]}\be$ whose homotopy classes are canonical. They will still be denoted by $\lambda $ and $\nu $.
\par\medskip 
The incidence varieties $\xna{m,n}$ are locally homeomorphic to the integrable model $U^{[m,n]}\be$ where $U$ is an open set of $\C^{2}$. This allows to put a stratification on $\xna{m,n}$. In this way, $\xna{m,n}$ is a stratified topological space locally modelled over an analytic space, so it has a fundamental homology class.
\par\medskip 
The construction by Nakajima and Grojnowski of representations of the Heisenberg super-algebra of $H\ee\be(X,\Q)$ into $\HHH:=\oplus_{n\in\N}\he H\ee\be\bigl( \xn,\Q\bigr)$ via correspondence actions of incidence varieties done in \cite{SchHilNa} and \cite{SchHilGr} also holds in the almost-complex setting:
\begin{theorem}\cite{SchHilGri}\label{Ch4suite2ThTrois} 
 If $(X,J)$ is an almost-complex compact fourfold, Nakajima operators $\mathfrak{q}_{i}\he(\alpha )$, $i\in\Z$, $\alpha \in H\ee\be(X,\Q)$, can be constructed.
They depend only on the deformation class of $J$ and satisfy the Heisenberg commutation relations\emph{:}
\[
\forall i,j\in\Z,\ \forall\alpha ,\beta\in H\ee\be(X,\Q),\quad
\bigl[ \mathfrak{q}_{i}\he(\alpha ),\mathfrak{q}_{j}\he(\beta )\bigr]=i\,\delta _{i+j,0}\he\,\Bigl( \int_{X}\alpha \beta\Bigr)\id_{\HHH}\he. 
\]
Furthermore, these operators induce an irreducible representation of $\hh\bigl( H\ee\be(X,\Q)\bigr)$ in $\HHH$ with highest weight vector $1$.
\end{theorem}
\section{The boundary operator}
The aim of this section is to carry out for symplectic fourforlds Lehn's computation of the boundary operator done in \cite{SchHilLe}. The first part of Lehn's argument can be adapted to the almost-complex case as in the proof of Nakajima relations done in \cite{SchHilGri}. This is explained in Section \ref{len}. The result of Section \ref{SecRef}, based on Donaldson's symplectic Kodaira's theorem allows to carry out in section \ref{exce} the second part of Lehn's argument when $X$ is symplectic.
\subsection{Lehn's formula in the almost-complex case}\label{len}
Let $(X,J)$ be an almost-complex compact fourfold and let $\T$ be the trivial complex line bundle on it. If $\jre$ is a relative structure in a \nbh\ of $Z_{n}\he$, there is a relative tautological bundle $\T^{[n]}_{\re}$ on $W^{[n]}_{\re}$ whose restriction to the fibers $\bigl( W^{[n]}_{\xb},\jrel{\xb}\bigr)$ are the usual tautological bundles $\bigl( \oo\he_{W^{[n]}_{\xb}}\bigr)^{[n]}\be$. If $\tn =\T^{[n]}_{\re}{}_{\vert\xn}\he$, $\tn $ satisfies the following properties (see \cite{SchHilGri}):
\begin{enumerate}
 \item [(i)] The class of $\tn $ in $K\bigl( \xn\bigr)$ is independent of $\jre$.
\item[(ii)] $-2c_{1}\he\bigl( \tn\bigr)$ is Poincar\'{e} dual with $\Q$-coefficients to $\bigl[ \partial\xn\bigr]$, where 
\[
\partial\xn=\bigl\{(\xi ,\xb)\ \textrm{in}\ \xn\ \textrm{such that there exists}\ p\ \textrm{in}\ \xb\ \textrm{with}\ \textrm{length}_{p}(\xi )\ge 2 \bigr\} 
\] is the so-called \emph{boundary of} $\xn$.
\item[(iii)] If $\apl{\lambda }{\xna{n+1,n}}{\xn}$ and $\apl{\nu }{\xna{n+1,n}}{\xna{n+1}}$ are the usual homotopy classes, $D$ is the exceptional divisor in $\xna{n+1,n}$ and $F$ is the complex line bundle on $\xna{n+1,n}$ whose first Chern class is Poincare dual to $-[D]$, then 
$\lambda \ee\be\,\T^{[n+1]}\be-\nu \ee\be\,\tn=F$ in $K\bigl( \xna{n+1,n}\bigr)$.
\end{enumerate}
We recall Lehn's definition of the boundary operator:
\begin{definition}\label{ah}
 Let $\HHH=\bigoplus_{n\ge 0}\he H\ee\be\bigl( \xn,\Q\bigr)$.
\begin{enumerate}
 \item [(i)] The \emph{boundary operator} $\apl{\mathfrak{d}}{\HHH}{\HHH}$ is defined by
$\mathfrak{d}\bigl[ (\alpha _{n}\he)_{n\ge 0}\he\bigr]=\bigl( c_{1}\he\bigl( \tn\bigr)\cup\alpha _{n}\he\bigr)_{n\ge 0}\he$.
\item[(ii)] If $A$ is an endomorphism of $\HHH$, \emph{the derivative} $A'$ of $A$ is defined by the formula
\[
A'=[\mathfrak{d},A]=\mathfrak{d}\circ A-A\circ\mathfrak{d}.
\]
\end{enumerate}
\end{definition}
\noindent We now state the following result:
\begin{theorem}\label{uu}
 Let $(X,J)$ be an almost-complex compact fourfold. There exist classes $\bigl( e_{n}\he\bigr)_{n\ge 0}\he$ in $H^{2}\be(X,\Q)$ such that
\[
\forall n,m\in\Z,\ \forall \alpha ,\beta \in H\ee\be(X,\Q),\quad
\bigl[ \mathfrak{q}'_{n}(\alpha ),\mathfrak{q}_{m}\he(\beta )\bigr]=-nm\,\mathfrak{q}_{n+m}\he(\alpha \beta )+\delta _{n+m,\,0}\he\Bigl( \int_{X}e_{|n|}\he\alpha \beta \Bigr)\id_{\HHH}\he.
\]
\end{theorem}
\begin{proof}
 Let $n,\,m$ be positive integers such that $m\ge n$, $\jrel{n\tim(m-n)}$ a relative integrable structure in a \nbh\ of $Z_{n\tim(m-n)}\he$, and let $I^{[m,\,n]}_{\,\T,\re}$ be the complex bundle on $W^{[m,\,n]}_{\re}$ whose restriction to the fibers $W^{[m,\,n]}_{\xb,\,\yb}$ are the kernels of the natural surjective morphisms from $\bigl( \oo_{W^{[m,\,n]}_{\xb,\,\yb}}\bigr)^{[m]}\be$ to $\bigl( \oo_{W^{[m,\,n]}_{\xb,\,\yb}}\bigr)^{[n]}\be$. We put $I_{\,\T}^{[m,\,n]}=I_{\,\T,\re}^{[m,\,n]}{}_{\vert\xna{m,\,n}}\he$. Then $I_{\,\T}^{[m,\,n]}=\nu \ee\be\,\T^{[n+1]}\be-\lambda \ee\be\,\tn$ in $K\bigl( \xna{m,\,n}\bigr)$, the proof being similar to  \cite[Proposition 3.5]{SchHilGri}. This allows to prove that correspondences behave well under derivations:
\begin{lemma}\label{hum}
 Let $u$ be a class in $H_{*}\he\bigl( \xna{m,\,n},\Q\bigr)$ and $\apl{u_{*}\he}{H\ee\be\bigl( \xn,\Q\bigr)}{H\ee\be\bigl( X^{[m]}\be,\Q\bigr)}$ be the associated correspondence map.\
Then $\bigl( u_{*}\he\bigr)'\be=\bigl[ u\cap c_{1}\he\bigl( I_{\,\T}^{[m,\,n]}\bigr)\bigr]_{*}\he$.
\end{lemma}
\begin{proof}
Let $\apl{\lambda }{\xna{m,\,n}}{\xn}$ and $\apl{\nu }{\xna{m,\,n}}{\xna{m}}$ be the usual maps. The identity 
\[
\nu \ee\be \T^{[m]}\be-\lambda \ee\be \tn =I_{\,\T}^{[m,\,n]}
\] is satisfied in $K\bigl( \xna{m,\,n}\bigr)$. Thus
\begin{align*}
 \bigl( u_{*}\he\bigr)'\tau &=c_{1}\he\bigl( \T^{[m]}\be\bigr)\cup u_{*}\he\tau-u_{*}\he\bigl( c_{1}\he\bigl(\T^{[n]}\be \bigr)
\cup \tau \bigr) \\
&=PD^{-1}\Bigl[ \bigl( \nu  _{*}\he(u\cap\lambda \ee\be\tau )\bigr)\cap c_{1}\he\bigl( \T^{[m]}\be\bigr)-\nu _{*}\he\bigl( u\cap\lambda \ee\be \bigl( c_{1}\he\bigl( \T^{[n]\be}\bigr)\cup\tau \bigr)\bigr)\Bigr]\\
&=PD^{-1}\Bigl[ \nu _{*}\he\Bigl( u\cap \bigl[\bigl( \nu \ee\be c_{1}\he\bigl( \T^{[m]}\be\bigr)-\lambda \ee\be c_{1}\he\bigl( \T^{[n]}\be\bigr)\bigr)\cup\lambda \ee\be\tau \bigr]\Bigr)\Bigr]\\
&=PD^{-1}\, \nu _{*}\he\Bigl( \bigl[ u\cap c_{1}\he\bigl(I_{\T}^{[m,\,n]} \bigr)\bigr]\cap \lambda \ee\be \tau \Bigr).
\end{align*}
\end{proof}
By this lemma, to prove the theorem it suffices to compute the commutator of two corres\-pondences. We adapt Lehn's proof exactly as in \cite{SchHilGri} for the Nakajima relations. This yields (see \cite{SchHilGriTh} for a detailed exposition):
\par\medskip 
$_{\ds *}$ For all $m,n$ in $\Z$ with $m\neq n$, $\bigl[ \mathfrak{q}'_{n}(\alpha ),\mathfrak{q}_{m}\he(\beta )\bigr]=\mu _{n,\,m}\he\,\mathfrak{q}_{n+m}\he(\alpha \beta )$, where $\mu _{n,\,m}\he\in\Z$.
\par 
$_{\ds *}$ For all $n$ in $\Z$, $\bigl[ \mathfrak{q}'_{n}(\alpha ),\mathfrak{q}_{-n}\he(\beta )\bigr]=\bigl( \int_{X}e_{|n|}\he\,\alpha \beta\bigr)\id_{\HHH}\he $, where the classes $e_{|n|}\he$ belong to $ H^{2}\be(X,\Q)$.
\par\medskip 
The terms $\mu _{n,\,m}\he$ and $e_{n}\he$ are \emph{the excess contributions}. The multiplicity $\mu _{n,\,m}\he$ can computed locally on $X$, so that Lehn's proof applies and gives $\mu _{n,\,m}\he=-nm$. However, this is not the case for the class $e_{n}\he$ which involves the \emph{global} geometry of $X$.
\end{proof}
\begin{corollary}\label{aas}
 When $\alpha $ runs through a basis of $H\ee\be(X,\Q)$, the operators $\mathfrak{d}$ and $\mathfrak{q}_{1}\he(\alpha )$ generate $\HHH$ from the vector $1$.
\end{corollary}
\begin{proof}
 This is a consequence of the relations $\bigl[ q'_{1}(\alpha ),q_{m}\he(1)\bigr]=-m\,q_{m+1}\he(\alpha )$, $m\in\N$.
\end{proof}

\subsection{Holomorphic curves in symplectic fourfolds}\label{SecRef}
Until now, we have only considered integrable structures in small open sets of $(X,J)$. To compute the excess classes $e_{n}\he$, we will construct pseudo-holomorphic curves in $X$ for perturbed almost-complex structures. To do so we use the following theorem of Donaldson \cite{SchHilDo}, which is a symplectic version of Kodaira's imbedding theorem:
\par\medskip 
\begin{theorem}\cite{SchHilDo}\label{ThDo}
 Let $(V,\omega )$ be a symplectic manifold of dimension $2n$ such that $\omega $ is an integral class. If $h$ is a lift of $\omega $ in $H^{2}\be(V,\Z)$, then, for $k\gg 0$, the classes $PD(kh)$ in $ H_{2n-2}\he(V,\Z)$ can be realized by homology classes of symplectic submanifolds. More precisely, if $J$ is an almost-complex structure on $V$ compatible with $\omega $, it is possible to write $PD(kh)=\bigl[ S_{k}\he\bigr]$ for $k$ large enough, where $S_{k}\he$ is a $J_{k}\he$-holomorphic codimension $2$ submanifold in $V$ and $||J_{k}\he-J||_{C^{0}\be}\he\le C/\sqrt{k}$.
\end{theorem}
\noindent We apply this theorem to our situation:
\begin{corollary}\label{cor}
 Let $(X,\omega )$ be a symplectic compact fourfold and $J$ an adapted almost-complex structure on $X$. Then there exist almost-complex structures 
$\bigl( J_{i}\he\bigr)_{1\le i\le N}\he$ arbitrary close to $J$ and $J_{i}\he$-holomorphic curves $\bigl( C_{i}\he\bigr)_{1\le i\le N}\he$ such that\emph{:}
\begin{enumerate}
 \item [(i)] For all $i$, $J_{i}\he$ is integrable in a \nbh\ of $C_{i}\he$.
\item[(ii)] The classes $\bigl[ C_{i}\he\bigr]$ span $H_{2}\he(X,\Q)$.
\end{enumerate}
\end{corollary}
\begin{proof}
 We pick $\alpha _{1}\he,\dots,\alpha _{N}\he$ in $H^{2}\be(X,\R)$ such that the $\omega +\alpha _{i}\he$'s are symplectic forms in $H^{2}\be(X,\Q)$ which generate $H^{2}\be(X,\Q)$. There exist almost-complex structures $\bigl( \ti{J}_{i}\he\bigr)_{1\le i\le N}\he$ adapted to 
$(\omega +\alpha _{i}\he)_{1\le i\le N}\he$ arbitrary close to $J$ if the $\alpha _{i}\he$'s are small enough. For suitable sufficiently large values of $m$, the symplectic forms $m(\omega +\alpha _{i}\he)$ are integral and Donaldson's theorem applies. We obtain \mbox{$J_{i}\he$-holomorphic} curves $(C_{i})_{1\le i\le N}\he$ where $J_{i}\he$ is arbitrary close to $\ti{J}_{i}\he$  and $\bigl[ C_{i}\he\bigr]=PD\bigl( k_{i}\he(\omega +\alpha _{i}\he)\bigr)$ in $H_{2}\he(X,\Q)$. 
Let $U_{i }\he$ be a small \nbh\ of the zero section of $N_{C_{i}\he/X}$. We identify $U_{i}\he$ with a \nbh\ of $C_{i}\he$ in X. 
Since $\dim C_{i}\he=2$, $J_{i}\he{}_{\vert \, C_{i}}$ is integrable. Since $C_{i}\he$ is $J_{i}\he$-holomorphic, $N_{C_{i}\he/X}$ is naturally a complex vector bundle over the complex curve $C_{i}\he$, so that we can put a holomorphic structure on it. This gives an integrable structure $J'_{i}$ on $U_{i}\he $ such that 
$J'_{i}{}_{\vert \, C_{i}}=J_{i}\he{}_{\vert \, C_{i}}$. We glue together $J_{i}\he$ and $J'_{i}$ in an annulus around $C_{i}\he$. The resulting almost-complex structure can be chosen arbitrary close to $J_{i}\he$ if $U_{i}\he $ is small enough.
\end{proof}
\subsection{Computation of the excess term in the symplectic case}\label{exce}
Our aim in this section is to prove the following theorem:
\begin{theorem}\label{qqq}
 If $(X,\omega )$ is a symplectic compact fourfold and $J$ is any  almost-complex structure compatible with $\omega $, the excess contributions $e_{n}\he$ of Theorem \ref{uu} are given by 
\[
\forall n\in\Z,\quad
e_{n}\he=\frac{1}{2}\,n^{2}(|n|-1)c_{1}\he(X).
\] This means that for all $n,m$ in $\Z$ and for all $\alpha ,\beta $ in $ H^{2}\be(X,\Q)$,
\[
\bigl[ \mathfrak{q}'_{n}(\alpha ),\mathfrak{q}_{m}\he(\beta )\bigr]=-nm\Bigl\{\mathfrak{q}_{n+m}\he(\alpha \beta )-\dfrac{|n|-1}{2}\,\delta _{n+m,0}\he\Bigl( \int_{X}c_{1}\he(X)\alpha \beta \Bigr)\id_{\HHH}\he\Bigr\}\cdot
\]
\end{theorem}
\begin{proof}
 If 
$X$ is a smooth projective surface, this theorem is due to Lehn (\cite{SchHilLe}, Th. 3.10). First we sketch his argument, then we adapt it in the symplectic case.
\par\medskip 
In this proof the notation $[Z]$ will be used to denote the cohomological cycle class of a cycle $Z$, and its the homology class as it was previously the case.
\par\medskip 
If $X$ is a smooth projective surface and $C$ is a smooth algebraic curve on $X$, results of Grojnowski and Nakajima (\cite{SchHilGr}, \cite{SchHilNabis}) describe explicitely the class 
$\bigl[ C^{[n]}\be\bigr]$ in $H^{2n}\be\bigl( \xn,\Q\bigr)$ in terms of the classes 
$\mathfrak{q}_{i_{1}\he}\he\bigl( [C]\bigr)\dots \mathfrak{q}_{i_{N}\he}\he\bigl( [C]\bigr)\,.\,1$, where $i_{1}\he,\dots,i_{N}\he$ are positive integers of total sum $n$. 
\par\medskip 
Let $X_{0}^{[n]}$ be the set of schemes in $\xn$ whose support is a single point. If $\partial C^{[n]}\be=C^{[n]}\be\cap\partial\xn$, the term $I=\ds\int_{\xn}\bigl[ X_{0}^{[n]}\bigr]\,.\,\bigl[ \partial C^{[n]}\be\bigr]$ can be computed in two different ways:
\par\medskip 
(i) The integral $I$ is equal to $\mathfrak{q}_{-n}\he(1)\,\bigl( \bigl[ \partial C^{[n]}\be\bigr]\bigr)$. Since $C^{[n]}\be$ and $\partial \xn$ intersect generically transversally, $\bigl[ \partial C^{[n]}\be\bigr]=\bigl[ \partial\xn\bigr]\,.\,\bigl[ C^{[n]}\be\bigr]=-2\,c_{1}\he(\T^{n}\be)\,.\,\bigl[ C^{[n]}\be\bigr]=-2\,\mathfrak{d}\bigl[C^{[n]}\be \bigr]$. Therefore, $I$ is a li\-near combination of terms of the form 
$\mathfrak{q}_{-n}\he(1)\,\he \mathfrak{q}_{i_{1}\he}\bigl( [C]\bigr)\dots \mathfrak{q}'_{i_{k}\he}\bigl( [C]\bigr)\dots \mathfrak{q}_{i_{N}\he}\bigl( [C]\bigr)\,.\,1$, where $i_{1}\he,\dots, i_{N}\he$ are positive integers of total sum $n$. These terms vanish except for:
\begin{enumerate}
 \item [--] $N=1$, $i_{1}\he=n$. 
Then $\mathfrak{q}_{-n}\he(1)\,\mathfrak{q}'_{n}([C])\,.\,1=\ds-\int_{X}e_{n}\he\,.\,[C]$.
\item[--] $N=2$, $i_{1}\he+i_{2}\he=n$. 
Then 
$\mathfrak{q}_{-n}\he(1)\,\mathfrak{q}_{k}([C])\,\mathfrak{q}'_{n-k}([C])\,.\,1=0$ and
\[
\mathfrak{q}_{-n}\he(1)\,\mathfrak{q}'_{k}([C])\,\mathfrak{q}_{n-k}\he([C])\,.\,1=-nk\,\mathfrak{q}_{k-n}\he([C])\,\mathfrak{q}_{n-k}\he([C])\,.\,1=nk(n-k)\,[C]^{2}\be.
\]
\end{enumerate}
This computation gives $I=\dfrac{1}{n}\ds\int_{X}e_{n}\,.\,[C]+\binom{n}{2}\,[C]^{2}\be$.
\par\medskip 
(ii) The cycle $C^{[n]}\be$ intersect transversally $X_{0}^{[n]}$ in its smooth locus and 
$C^{[n]}\be\cap X_{0}^{[n]}=C_{0}^{[n]}\simeq C$. Therefore
$
I=\int_{\xn}\bigl[ X_{0}^{[n]}\bigr]\,.\,\bigl[ C^{[n]}\be\bigr]\,.\,\bigl[ \partial\xn\bigr]=\deg_{\,C}\he\bigl[ \oo_{\xn}\he\bigl( \partial\xn\bigr)\bigr]=\deg
_{\,C}\he\bigl[ \oo_{C^{[n]}\be}\he\bigl( \partial C^{[n]}\be\bigr)\bigr]$,
which is $-n(n-1)\,\deg_{\,C}\he K_{X}\he$ by direct computation.
\par\medskip 
In the algebraic case, the excess terms $e_{n}\he$ lie in the Neron-Severi group of $X$ so that it is enough to prove that for every smooth algebraic curve $C$,
$\ds\int_{X}\Bigl[ e_{n}\he-\dfrac{1}{2}\,n^{2}\be(n-1)\,c_{1}\he(X)\Bigr]\,.\,[C]=0$. 
This is proved by comparison of the two expressions obtained for $I$.
\par\medskip 
Let us now suppose that $(X,\omega )$ is a symplectic compact fourfold and that $J$ is an adapted almost-complex structure on $X$. If $\gamma \in H\ee\be(X)$ is a class
of even degree, we define the vertex operators $\bigl( S_{m}\he(\gamma )\bigr)_{m\ge 0}\he$ by the formula
$\sum_{m\ge 0}S_{m}\he(\gamma )\, t^{m}\be =\exp\left(\sum_{n>0}\frac{(-1)^{n-1}}{n}\, \mathfrak{q}_{n}^{\vphantom{A}}(\gamma )\, t^{n}\be\right)$. Since $\gamma $ is of even degree, the operators $\bigl( \mathfrak{q}_{i}\he(\gamma )\bigr)_{i>0}\he$ commute in the usual sense, and the definition of $S_{m}\he(\gamma )$ makes sense.
\begin{lemma}\label{sans} \emph{(\cite{SchHilGr}, \cite{SchHilNabis} in the integrable case).} Let $\ti{J}$ be an almost-complex structure close to $J$, and let $C$ be a $\ti{J}$-holomorphic curve. Suppose that $\ti{J}$ is integrable in a \nbh\ of $C$. Then $\bigl[ C^{[n]}\be\bigr]=S_{n}\he\bigl( [C]\bigr)\, .\, 1$. 
\end{lemma}
\begin{proof}
Since the Nakajima operators are invariant by deformation of the almost-complex structure, we can make the computations with  almost-complex structures equal to $\ti{J}$ when all the points are near $C$.
 Let $U$ be a small \nbh\ of $C$ such that $\ti{J}$ is integrable in $U$. Then, for any positive integers $n$ and $m$ such that $m>n$, the Hilbert schemes $\xna{m}$, $\xn$ and the incidence variety $\xna{m\!,\,n}$ are the usual integrable ones over $S^{m}\be U$, 
$S^{n}\be U$ and $S^{n}\be U\tim S^{n-m}\be U$ respectively. 
Since Lemma \ref{sans} holds in $H^{2n}_{c}\bigl( U^{[n]}\be,\Q\bigr)$, 
we are done. 
\end{proof}
If $(C,\ti{J})$ satisfies the hypotheses of Lemma \ref{sans}, Lehn's computations recalled above apply verbatim and give $\ds\int_{X}\Bigl[ e_{n}\he-\dfrac{1}{2}\,n^{2}\be(|n|-1)\,c_{1}\he(X)\Bigr]\,.\,[C]=0$. By Corollary \ref{cor}, $H^{2}\be(X,\Q)$ is spanned by cohomology classes of such holomorphic curves. This gives the result.
\end{proof}
The derivative of the Nakajima operators can be explicitely computed in terms of the Virasoro operators $\mathfrak{L}_{n}\he(\alpha )$ defined in \cite[Section 3.1]{SchHilLe}:
\begin{corollary}\label{www}
If $(X,\omega )$ is a symplectic compact fourfold and $J$ is a compatible almost-complex structure, then for all $n$ in $\Z$, $\mathfrak{q}'_{n}(\alpha )=n\, \mathfrak{L}_{n}\he(\alpha )-\frac{1}{2}\,n(|n|-1)\mathfrak{q}_{n}^{\vphantom{A}}\bigl( c_{1}\he(X)\,\alpha \bigr)$.
\end{corollary}
\noindent For the proof, see \cite[p. 180]{SchHilLe}. 
\section{The ring structure of $H^{*}\bigl(X^{[n]},\Q\bigr)$}
\subsection{Geometric tautological Chern characters}
Let $(X,J)$ be an almost-complex compact fourfold, 
$
\apl{\lambda }{\xna{n+1}}{\xn,}\ \apl{\nu }{\xna{n+1,n}}{X^{[n+1]}\be,}\ \apl{\rho  }{\xna{n+1,n}}{X}
$
the three associated maps, which are canonical up to homotopy, and $D\suq \xna{n+1,n}$ the exceptional divisor.
\par\medskip 
If $E$ is a complex vector bundle on $X$, it is possible to associate to $E$ a sequence of tautological vector bundles $\bigl( E^{[n]}\bigr)_{n>0}\he$ on $\xn$. These tautological bundles are constructed in \cite{SchHilGri} using relative holomorphic structures on $E$, and their classes in \mbox{$K$-theory} are shown to be independent of these auxiliary structures. This defines tautological morphisms from $K(X)$ to $K\bigl( \xn\bigr)$.
\par 
Let $F$ be the complex line bundle on $\xna{n+1,n}$ such that $c_{1}\he(F)$ is Poincar\'{e} dual to $-[D]$. Then (see \cite[3.2]{SchHilGri}),
 $\nu \ee\be E^{[n+1]}\be=\lambda \ee E^{[n]}\be+\rho \ee\be E\oti $F
in $K\bigl( \xna{n+1,n}\bigr)$. This gives the relation $\nu \ee\be\bigl( \ch\bigl( E^{[n+1]}\be\bigr)\bigr)=\lambda \ee\be\bigl( \ch\bigl( E^{[n]}\be\bigr)\bigr)+\rho \ee\be \ch(E)\,.\,c_{1}\he(F)$ in $H^{\textrm{even}}\be\bigl( X^{[n+1,n]}\be,\Q\bigr)$. 
\begin{lemma}\label{Ch4SuiteLemUn}
 For every class $\alpha $ in $H^{\emph{even}}\be(X,\Q)$ and every $n$ in $\N\ee\be$, there exists a unique class $G(\alpha ,n)$ in $H^{\emph{even}}\be\bigl( \xn,\Q\bigr)$ such that $G(\alpha ,1)=\alpha $, and for all $n$ in $\N\ee\be$, 
\[\nu \ee\be G(\alpha ,n+1)-\lambda \ee\be G(\alpha ,n)=\rho \ee\be \alpha \,.\,c_{1}(F).\]
\end{lemma}
\begin{proof}
By the degeneration of the Atiyah-Hirzebruch spectral sequence from \mbox{$K$-theory} to coho\-mology with \mbox{$\Q$-coefficients}, we have isomorphisms $\apliso{\ch}{K\bigl( \xn\bigr)\oti_{\Z}\he\Q}{H^{\textrm{even}}\be\bigl( \xn,\Q\bigr).}$ Therefore, we can define the classes $G(\alpha ,n)$ in $H^{\textrm{even}}\bigl( \xn,\Q\bigr)$ as follows: if $E$ is the unique class in $K(X)\oti_{\Z}\he\Q$ such that $\ch(E)=\alpha $, then $G(\alpha ,n)=\ch\bigl( E^{[n]}\be\bigr)$. Furthermore, $G(\alpha ,n)$ is unique since $\nu  _{*}\he\nu \ee\be=\frac{1}{n+1}\id$.
\end{proof}
\subsection{Virtual tautological Chern characters}\label{decadix}
This section is somehow technical and can be omitted if we suppose that $b_{1}\he(X)=0$, that is if $H^{\textrm{odd}}\be(X,\Q)=0$. The purpose here is to extend Lemma \ref{Ch4SuiteLemUn} to odd-dimensional cohomology classes. We adapt the method originally developed in the projective case by Li, Qin and Wang in \cite{SchHilLQWMA}.
\begin{proposition}\label{Ch4SuitePropUn}
 For every class $\alpha $ in $H\ee\be(X,\Q)$ and every $n$ in $\N\ee\be$, there exists a unique class $G(\alpha ,n)$ in $H\ee\be\bigl( \xn,\Q\bigr)$ such that $G(\alpha ,1)=\alpha $ and for all $n\in\N\ee\be$, 
\[\nu \ee\be G(\alpha ,n+1)-\lambda \ee\be G(\alpha ,n)=\rho \ee\be \alpha \,.\,c_{1}(F).\]
\end{proposition}
\begin{remark}\label{Ch4SuiteRemUn}
If $X$ is projective and $\yn\suq\xn\tim X$ is the incidence locus, it is proved in \cite{SchHilLQWMA} that
$G(\alpha ,n)=\pr_{1}\ee\bigl[ \ch(\oo_{\yn}\he)\,.\,\pr_{2}\ee\alpha \,.\,\pr_{2}\ee\,\td(X)\bigr]$.
\end{remark}
\begin{proof}
 We use the relative trick and the machinery of relative coherent sheaves developed in the appendix. Since these methods are going to be explained thoroughly in the second part of the paper with similar computations (e.g. in Section \ref{qwe}), we skip some lengthy details.
\par\medskip 
If $\xg$ is a \rsas\ and $\zg$ a relative analytic subspace of $\xg$ (see Definition \ref{DefUnAppenCh4}), we will denote $\lim\limits_{\genfrac{}{}{0pt}{}{\longleftarrow}{\xgp}}\he H\ee_{\xgp\cap\zg}(\xgp,\Q)$ by 
$\hf _{\zg}\ee(\xg,\Q)$, where in the projective limit $\xgp$ runs through all open relatively compact relative analytic subspaces of $\xg$.
\par 
If $W$ is a \nbh\ of $Z_{n}\he$ in $\snx\tim X$, $\yg_{n}\he$ is the relative incidence locus and $\ore{n}$ is the relative incidence \sh\ on $W^{[n]}_{\re}\tim_{\snx}\he W$ (see Definition \ref{a} and Definition \ref{b}), the topological class of $\ore{n}$ lies in
$\lim\limits_{\genfrac{}{}{0pt}{}{\longleftarrow}{\xgp}}\he K_{\xgp\cap\yg_{n}\he}\he(\xgp)$. Let $\apl{\pi }{W}{X}$, $\apl{p}{W^{[n]}_{\re}\tim_{\snx}\he W}{W_{\re}^{[n]}}$ be the natural projections.
We define the following cohomology classes:
\begin{enumerate}
 \item [$_{\ds *}$] $\mu ^{\re}_{n}$ is the Chern character of $\ore{n}$ in $\hf\ee_{\yg_{n}\he}\bigl( W^{[n]}_{\re}\tim_{\snx}\he W,\Q\bigr)$.
\item[$_{\ds *}$] $G(\alpha ,n)^{\re}\be=p_{*}\he\bigl[ \mu _{n}^{\re}\,.\,\pi \ee\be\alpha \,.\,\pi \ee\be\td(X)\bigr]$ in $\hf\ee\be\bigl( W^{[n]}_{\re},\Q\bigr)$.
\item[$_{\ds *}$] $G(\alpha,n)=G(\alpha ,n)_{\vert\xn}^{\re}$.
\end{enumerate}
We take the notations introduced at the beginning of Section \ref{Comparaison}, so that $W$ will be from now on a \nbh\ of 
$Z_{n\tim 1}$ in $\snx\tim X$. 
Let $\apl{\pi }{W}{X,}$ $\apl{p}{W^{[n+1,n]}_{\re}\tim_{\snx\tim X}\he W}{W^{[n+1,n]}_{\re}}$ and $\apl{q}{W^{[n]}_{\re}\tim_{\snx\tim X}\he W}{W_{\re}^{[n]}}$ be the natural projections. We define new cohomology classes:
\begin{enumerate}
 \item [$_{\ds *}$] $\ti{\mu }_{n}^{\re}=\ch\bigl( \ti{\oo}^{\re}_{n}\bigr)$ in $\hf\ee_{\,\ti\yg_{n}\he}\bigl( W^{[n]}_{\re}\tim_{\snx}\he W,\Q\bigr)$, where $\ti{\yg}_{n}\he=\supp\bigl( \ti\oo^{\re}_{n}\bigr)$.
\item[$_{\ds *}$] $\ti{G}(\alpha ,n)^{\re}\be=q_{*}\bigl[ \ti{\mu }_{n}^{\re}\,.\,\pi \ee\be\alpha \,.\,\pi \ee\be\td(X)\bigr]$.
\end{enumerate}
Then $\psi _{W}\ee\, \ti{G}
(\alpha ,n+1)^{\re}\be-\phi \ee_{W}\ti{G}(\alpha ,n)^{\re}\be=p_{*}\he\bigl[ (\psi _{W}\ee\, \ti{\mu }^{\re}_{n+1}-\phi \, \ee_{W}\, \ti{\mu }_{n}^{\re})\,.\,\pi \ee\be\alpha \,.\,\pi \ee\be\td(X)\bigr]$.
Lemma \ref{f} shows that this quantity is equal to $p_{*}\bigl[ p\ee\be\ch{(\LL)}\,.\,\rho \ee_{\re,\,W}\ch\bigl( \ore{\Delta _{\re}\he}\bigr)\,.\,\pi \ee\be\alpha \,.\,\pi \ee\be\td(X)\bigr]$, which is equal to $\ch(\LL)\,\rho _{\re}\ee\, r_{*}\he\bigl[ \ch\bigl(\ore{\Delta _{\re}\he}\bigr)\,.\,\pi \ee\be\alpha \,.\,\pi \ee\be\td(X)\bigr]$ via the second diagram of the proof of Proposition \ref{SansLabBis}. Using the same argument as in proof of Lemma \ref{EtToc} (iv), 
\[
\bigl[ \psi _{W}\ee\,\ti{G}(\alpha ,n+1)^{\re}\be-\phi _{W}\ee\,\ti{G}(\alpha ,n)^{\re}\be\bigr]_{\vert\xna{n+1,n}}\he=\ch(\LL)_{\vert\xna{n+1,n}}\he\,\rho \ee\be\bigl[ \pr_{1*}\he\ch\bigl( \ci_{\Delta _{X}\he}\bigr)\,.\,\alpha \,.\,\td(X)\bigr]
\]
which is equal, by the Grothendieck-Riemann-Roch theorem of \cite{SchHilAtHi} applied to the diagonal injection, to $c_{1}\he(F)\,.\,\rho \ee\be(\alpha )$. To conclude the proof, it suffices to show that 
\[
\psi \ee_{W}\,\ti{G}(\alpha ,n+1)^{\re}\be{}\he_{\vert\xna{n+1,n}}=\nu \ee\be \,G(\alpha ,n+1)\quad\textrm{and}\quad\phi \ee_{W}\,\ti{G}(\alpha ,n)^{\re}\be{}\he_{\vert\xna{n+1,n}}=\lambda \ee\be G(\alpha ,n).
\]
This is performed exactly as in Lemma \ref{EtToc}, (i)--(iii).
\end{proof}
\subsection{The ring structure and the crepant resolution conjecture}
In this section, $X$ will be a symplectic compact fourfold endowed with a compatible almost-complex structure.
\par\medskip 
We introduce operators acting on $\HHH=\bop\nolimits_{n\in\N}\he H\ee\be\bigl( \xn,\Q\bigr)$ by cup-product with the components of the virtual tautological Chern characters constructed in Section \ref{decadix}.
\begin{definition}\label{5}
 Let $\alpha \in \hb(X,\Q)$ be a homogeneous cohomology class. Then
\begin{enumerate}
 \item [(i)] $G_{i}\he(\alpha ,n)$ denotes the $(|\alpha |+2i)$-th component of $G(\alpha ,n)$.
\item[(ii)] $\mathfrak{S}_{i}\he(\alpha )$ denotes the operator $\HHH$ which acts on $\hb(\xn,\Q)$ by cup product with 
$G_{i}\he(\alpha ,n)$.
\end{enumerate}
\end{definition}
Then the following result, originally proved by Lehn for geometric tautological Chern characters and generalized by Li, Qin and Wang for the virtual ones, is still valid: 
\begin{proposition}\label{UN}
For all $\alpha ,\beta $ in
 $\hb(X,\Q)$ and for all $k$ in $\N$, $\bigl[ \mathfrak{S}_{k}\he(\alpha ),\mathfrak{q}_{1}\he(\beta )\bigr]=\frac{1}{k!}\, \mathfrak{q}_{1}^{(k)}(\alpha \beta )$.
\end{proposition}
\begin{proof}
 This is a consequence of Lemma \ref{hum} and Proposition \ref{Ch4SuitePropUn} (see \cite[Theorem 4.2]{SchHilLe}).
\end{proof}
We can now state some significant results on the cohomology rings of Hilbert schemes of symplectic fourfolds. These results are known in the integrable case and are formal consequences of the various relations between $\mathfrak{q}_{n}^{\vphantom{A}}(\alpha )$, $\mathfrak{d}$, $\mathfrak{L}_{n}\he(\alpha )$ and $\mathfrak{S}_{i}\he(\alpha )$ (see e.g. Theorem 2.1 in \cite{SchHilLQW2003}), even though there is a lot of nontrivial combinatorics involved in the proofs. Thus the following results are formal consequences of Theorem \ref{qqq}, Corollary \ref{www} and Proposition \ref{UN}.
\begin{theorem}\label{6}
If $0\le i<n$ and $\alpha $ runs through a fixed basis of $\hb(X,\Q)$, the classes $G_{i}\he(\alpha ,n)$ generate the ring $\hb\bigl( \xn,\Q\bigr)$.
\end{theorem}
This result was initially proved in \cite{SchHilLQWMA} using vertex algebras. For other proofs, see \cite{SchHilLQW2001} and \cite{SchHilLQW2003}.
\begin{theorem}\label{asd}
 For every integer $n$, the ring $H\ee\be\bigl( \xn,\Q\bigr)$ can be built by universal formulae from the ring $H\ee\be(X,\Q)$ and the first Chern class of $X$ in $H^{2}\be(X,\Q)$.
\end{theorem}
\noindent For the proof as well as an effective statement, see \cite{SchHilLQW2003}.
\par\medskip 
There is a geometrical approach to the ring structure of $\hb\bigl( \xn,\Q\bigr)$ through orbifold cohomology. If $J$ is an adapted almost-complex structure on $X$, $\snx$ is an almost-complex Gorenstein orbifold. We can therefore consider the associated Chen-Ruan (or \emph{orbifold}) cohomology ring $H\ee_{\textrm{CR}}\bigl( \snx,\Q\bigr)$ which is $\Z$-graded  and depends only on the deformation class of $J$ (see \cite{SchHilCR}, \cite{SchHilALR}, \cite{SchHilFG}). 
\par\medskip 
After works by Lehn-Sorger and Li-Qin-Wang, Qin and Wang developed a set of axioms which characterize $H^{*}_{\textrm{CR}}\bigl( \snx,\Q\bigr)$ as a ring (see \cite{SchHilALR}):
\begin{theorem}\cite{SchHilQW}\label{QuinWang}
Let $A$ be a graded unitary ring and $(X,J)$ an almost-complex compact fourfold. We suppose that
\begin{enumerate}
 \item [(i)] $A$ is an irreducible $\hh\bigl( \hb(X,\C)\bigr)$-module and $1$ is a highest weight vector.
\item[(ii)] For all $\alpha $ in $\hb(X,\C)$ and for all $i$ $\N$, there exist classes $O_{i}\he(\alpha ,n)\in A^{|\alpha |+2i}\be$ such that if $\mathfrak{D}_{i}\he(\alpha )$ is the operator of product by 
$\bop_{n}\he O_{i}\he(\alpha ,n)$ and $\mathfrak{D}_{1}\he(1)=\mathfrak{d}$ is the derivation,
\begin{enumerate}
\item[(1)] $\forall \alpha ,\beta \in\hb(X,\C)$, $\forall k\in\N$,\quad $\bigl[ \mathfrak{D}_{k}\he(\alpha ),\mathfrak{q}_{1}\he(\beta  )\bigr]=\mathfrak{q}_{1}^{(k)}(\alpha \beta )$.
\item[(2)] If $\delta _{X}\he$ is the class in 
$H\ee\be(X,\Q)^{\oti 3}\be$ mapped by the K\"{u}nneth isomorphism to the cycle class of the diagonal in $X^{3}\be$, $\ds\sum_{l_{1}\he+l_{2}\he+l_{3}\he=0}:\mathfrak{q}_{l_{1}\he} \mathfrak{q}_{l_{2}\he} \mathfrak{q}_{l_{3}\he}\!\!:(\delta _{X}\he)=-6\, \mathfrak{d}$. 
\end{enumerate}
\end{enumerate}
 Then $A$ is isomorphic as a ring to  $H\ee_{\emph{CR}}\bigl( \snx,\C\bigr)$.
\end{theorem}
In $(2)$, we used the physicists' normal ordering convention
\[
:\mathfrak{q}_{l_{1}\he} \mathfrak{q}_{l_{2}\he} \mathfrak{q}_{l_{3}\he}\!\!:\,=
\mathfrak{q}_{l_{1}'}\he \mathfrak{q}_{l_{2}'}\he \mathfrak{q}_{l_{3}'}\he,\ \textrm{where}\ 
\{l_{1}\he,l_{2}\he,l_{3}\he\}=\{l_{1}',l_{2}',l_{3}'\}\ \textrm{and}\ l'_{1}\le l_{2}'\le l_{3}'.
\]
We apply this theorem to prove Ruan's conjecture for the symmetric products of a symplectic fourfold with torsion first Chern class.
\begin{theorem}\label{TROIS}
 Let $(X,\omega )$ be a symplectic compact fourfold with vanishing first Chern class in $H^{2}\be(X,\Q)$. Then Ruan's crepant conjecture holds for $\snx$, i.e. the rings $\hb\bigl( \xn,\Q\bigr)$ and $H\ee_{\emph{CR}}\bigl( \snx,\Q\bigr)$ are isomorphic.
\end{theorem}
\begin{proof}
 Let $O_{k}\he(\alpha ,n)=k!\, \mathfrak{S}_{k}\he(\alpha ,n)$. Then (1) is exactly Proposition \ref{UN}. The relation (2) is a formal consequence of the Nakajima relations and of the formulae $\bigl[ \mathfrak{q}_{n}'(\alpha ),\mathfrak{q}_{m}^{\vphantom{A}}(\beta )\bigr]=-nm\, \mathfrak{q}_{n+m}^{\vphantom{A}}(\alpha \beta )$, 
$\mathfrak{q}'_{n}(\alpha )=n\, \mathfrak{L}_{n}\he(\alpha )$. 
\end{proof}

\section{The cobordism class of $X^{[n]}$}\label{SecCinq}
In this section, $(X,J)$ is an almost-complex compact fourfold, and no symplectic hypotheses are required. The almost-complex Hilbert schemes $\xn$ are endowed with a stable almost complex structure (see \cite{SchHilVo1}), hence define almost-complex cobordism classes. By a fundamental result of
Novikov \cite{SchHilNo} and Milnor \cite{SchHilMi}, the almost-complex cobordism class of $\xn$ is completely determined by the Chern numbers $\ds\int_{\xn}P\bigl[ c_{1}\he(\xn),\dots,c_{2n}\he(\xn)\bigr]$ where $P$ runs through all polynomials $P$ in $\Q\bigl[ T_{1}\he,\dots,T_{2n}\he\bigr]$ of weighted degree $4n$, each variable $T_{k}\he$ having degree $2k$. We intend to prove the following result:
\begin{theorem}\label{QUATRE}
 The almost-complex cobordism class of $\xn$ depends only on the almost-complex cobordism class of $X$.
\end{theorem}
This means that if $P$ is a weighted polynomial in $\Q\bigl[ T_{1}\he,\dots,T_{2n}\he\bigr]$ of degree $4n$, there exists a weighted polynomial $\ti{P}\bigl[ T_{1}\he,T_{2}\he\bigr]$ of degree $4$, depending only on $P$ and $n$, such that 
\[
\int_{\xn}P\bigl[ c_{1}\he(\xn),\dots,c_{2n}\he(\xn)\bigr]=\int_{X}\ti{P}\bigl[ c_{1}\he(X),c_{2}\he(X)\bigr].
\]
This result has been proved by Ellinsgrud, G\"{o}ttsche and Lehn in 
\cite{SchHilEGL} when $X$ is projective. Let us describe briefly the main steps of the argument and the tools we need.
\par\medskip 
Let $J_{\vphantom{n}}^{\re}$ be a relative integrable structure in a \nbh\ $W$ of $Z_{n}\he$. Then the relative Hilbert scheme $\bigl( W^{[n]}_{\re},\jre\bigr)$ is fibered in smooth analytic spaces over $\snx$, so that we can consider its relative tangent bundle $T^{\re}\be\,W_{\re}^{[n]}$ which is a continuous complex vector bundle on $W^{[n]}_{\re}$. It will turn out that the class $\bigl[ T^{\re}\be\,W_{\re}^{[n]}\bigr]_{\vert\xn}\he$ in $K(\xn)$ is exactly the class of the tangent bundle $T\xn$, where $\xn$ is endowed with the differentiable and stable almost-complex structures constructed by Voisin in \cite{SchHilVo1}. Therefore, the \mbox{$K$-theory} class of $T\xn$ can be understood via the tangent bundle of classical Hilbert schemes.  
\par\medskip 
The main idea in \cite{SchHilEGL} is to relate $X^{[n+1]}\be$ and $\xn\tim X$ via the smooth incidence Hilbert scheme $X^{[n+1,n]}\be$. The essential step in their argument is the explicit comparison in $K\bigl( X^{[n+1,n]}\be\bigr)$ of the classes of $T\xn$ and $T X^{[n+1]}\be$. This is carried out using the explicit description of $T\xn$ as $\pr_{1*}\he\hh om\bigl( \jj_{n}\he,\oo_{n}\he\bigr)$, where $\jj_{n}\he$ is the ideal sheaf of the incidence locus $\yn \suq\xn\tim X$ and $\oo_{n}\he$ is the structure sheaf of $\yn $. So it appears necessary to consider coherent sheaves and not only locally free ones. 
\par\medskip 
In the almost-complex setting, the coherent sheaves have no equivalent. Instead of working directly on the almost-complex Hilbert schemes and the associated incidence varieties, we use the corresponding relative objects, which can be considered as homotopically equivalent to the original ones, but possess a much stronger structure: they are fibered in analytic spaces over a singular basis. Each fiber consists of the initial object (Hilbert scheme, incidence variety,...) associated to an open set of $X$ with an integrable structure on it.
The almost-complex Hilbert scheme $\xn$ will be for instance replaced by the fibration $W^{[n]}_{\re}$ over $\snx$ associated to a relative integrable complex structure $\jre$: the corresponding fibers are the integrable Hilbert schemes $\bigl( W_{\xb}^{[n]},J^{\re}_{\xb}\bigr)_{\xb\in\snx}\he$.
\par\medskip 
In the appendix (Section \ref{AppendCh4}), we develop a general formalism for relative coherent sheaves on spaces $\xg/B$ fibered in smooth analytic sets over a differentiable basis $B$ with quotient singularities, such as $W^{[n]}_{\rel}$. These \rsas s carry a sheaf $\oo_{\xg}^{\, \rel}$ which is the sheaf of smooth functions holomorphic in the fibers (see Definition \ref{DefUnAppenCh4}).
\par
Intuitively, a relatively coherent sheaf $\ff$ on a \rsas\ $\xg$ over $B$ is a family of coherent sheaves
$\bigl( \ff_{b}\he\bigr)_{b\in B}\he$ on $\bigl( \xg_{b}\he\bigr)_{b\in B}\he$ varying smoothly with $b$ and locally trivially on $\xg$. 
If we take relative holomorphic coordinates on $\xg$, the local model for $\xg$ is $Z\tim V$, where $Z$ is a smooth analytic set and $V$ is an open subset of the base $B$. Then the local model for a relatively coherent sheaf on $Z\tim V$ is $\pr_{1}^{-1}\g\oti_{\pr_{1}^{-1}\oo_{Z}\he}\he \ore{Z\tim V}$, where $\g$ is a coherent analytic sheaf on $Z$  (see Definition \ref{DefQuatreAppendCh4}). 
\par
The usual operations on coherent sheaves (such as internal $\textrm{Hom}$, tensor product, dual,
pull-back, push-forward and the associated derived operations) can be performed on relatively cohe\-rent sheaves for smooth morphisms holomorphic in the fibers satisfying some triviality conditions (see Definitions \ref{DefDeuxAppenCh4} and \ref{DefSixAppendCh4}). For the push-forward, we will only have to deal with the situation where the map is finite on the support of the \sh, which is technically much simpler than the general case. 
\par
In this context, it is possible to construct a relative version of the usual analytic $K$-theory for coherent \sv\ (see Definitions \ref{DefNeufAppendCh4} and \ref{DefUnBisAppendCh4}) as well as associated operations.
\par
This being done, it will be necessary to consider relatively coherent sheaves as elements in topological \mbox{$K$-theory}. This is achieved by the following proposition:
\begin{proposition}\label{SansLab}
If $\ff$ is a relatively coherent sheaf on a relative smooth analytic space $\xg$ over $B$ and $\xg'$ is relatively compact in $\xg$, then $\ff^{\, \infty }:=\ff\oti_{\ore{\xg}}\he\ci_{\xg}$ admits a resolution on 
$\xg'$ by complex vector bundles $\bigl( E_{i}\he\bigr)_{1\le i\le N}\he$. Besides, the element 
$[\ff^{\, \infty }\be]:=\sum_{i=1}^{N}(-1)^{i-1}[E_{i}\he]$ in $K(\xg')$ is independent of 
$E_{\bullet}\he$.
\end{proposition}
According to this proposition, it is possible to associate to any relatively coherent \sh\ $\ff$ on a \rsas\ $\xg$ a \emph{topological class} $[\ff^{\,\infty }\be]$ in $\ds\lim_{\genfrac{}{}{0pt}{2}{\longleftarrow}{\xgp\subset\subset \xg}}K(\xgp)$ and therefore Chern classes in $\ds\lim_{\genfrac{}{}{0pt}{2}{\longleftarrow}{\xgp\subset\subset \xg}}{H}\ee\be(\xgp,\Z)$. Besides, the class $[\ff^{\,\infty }\be]$ depends only on the relative class of $\ff$ in $\kre{}(\xg)$ by Proposition \ref{Compl} (i).
\par\medskip 
This device enables us to carry out the proof of \cite{SchHilEGL} in a relative context.
\subsection{The relative incidence \sh}
In this section, we introduce the relative incidence sheaf $\ore{n}$ and we compute its Chern classes. We will use the following notations.
\begin{enumerate}
 \item [--] $W$ is a small \nbh\ of the incidence locus $Z_{n}\he$ in $\snx\tim X$.
\item[--] $\jrel{n}$ is a relative integrable structure on $W$ parametrized by $\snx$.
\end{enumerate}
\begin{definition}\label{a}
 The \emph{relative incidence locus} $\ygn $ is the relative singular analytic subspace of $W_{\re}^{[n]}\tim_{\snx}\he W$ defined by
\[
\ygn=\bigl\{(\xi,\,w,\,\xb)\ \textrm{in}\ W_{\re}^{[n]}\tim_{\snx}\he W\ \textrm{such that}\  w\in \supp(\xi)\bigr\}\cdot
\]
The fibers $\bigl( \yg_{n,\,\xb}\he\bigr)_{\xb\in\snx}\he$ are the usual incidence loci in $W_{\xb}^{[n]}\tim W_{\xb}\he$.
\end{definition}
\begin{definition}\label{b}
 The \emph{relative incidence \sh} $\ore{n}$ is the relatively coherent \sh\ $\ore{\ygn }$ on $W_{\re}^{[n]}\tim_{\snx}\he W$. The associated ideal \sh\
$\jjrel{ \ygn }$ will be denoted by $\jjrel{n}$.
\end{definition}
By Proposition \ref{SansLab}, $\ore{n}$ defines a topological class in
$\ds\lim_{\genfrac{}{}{0pt}{2}{\longleftarrow}{\xgp\subset\subset W_{\re}^{[n]}\tim_{\snx}\he W}}K(\xgp)$ and therefore by restriction a class in $K\bigl( \xn\tim X\bigr)$ which is independent of $J^{\re}_{n}$ by Proposition \ref{Compl} (ii).
\begin{definition}\label{c}
 The classes $\mu _{i,\,n}\he$ in $H^{2i}\be\bigl( \xn\tim X,\Q\bigr)$ are defined by
$\mu _{i,\,n}\he=c_{i}\he\bigl( \ore{n}\bigr)_{\vert \xn\tim X}\he$.
\end{definition}
-- Let $\ti W$ be a small \nbh\ of the incidence locus $Z_{n,\,1}\he$ in $\snx\tim X$ and $\jrel{n,1}$ a relative integrable structure parametrized by $\snx\tim X$. For $(\xb,p)\in \snx\tim X$, $\tw^{[n+1,n]}_{\xb,\,p}$ is smooth (\cite{SchHilCh}, \cite{SchHilTi}). Therefore, $\tw^{[n+1,n]}_{\re}$ is a relative smooth analytic space over $\snx\tim X$.
\par\medskip 
-- Let $D^{\re}\be$ be the relative exceptional divisor in $\tw^{[n+1,n]}_{\re}$ defined by 
\[
D^{\re}\be=\bigl\{(\xi,\,w,\,\xb,\,p)\ \textrm{in}\ \tw^{[n+1,n]}_{\re}\ \textrm{such that}\  w\in\supp(\xi )\bigr\}\cdot
\]
Then $D^{\re}\be$ is a relative singular analytic divisor of $\tw^{[n+1,n]}_{\re}$, and we get an associated relative holomorphic line bundle
$\oo^{\re}\be\bigl( -D^{\re}\be\bigr)$ on $\tw_{\re}^{[n+1,n]}$. Let $\apl{\rho }{X^{[n+1,n]}\be}{X}$ be the residual morphism and $\apl{\sigma }{X^{[n+1,n]}\be}{\xn\tim X}$ be the morphism $(\lambda ,\rho )$.
\begin{definition}\label{l}
 We define a class $l$ in $H^{2}\bigl(X^{[n+1,n]}\be,\Q\bigr)$ by $l=c_{1}\he\bigl[ \ore{}(-D_{\re}\he)\bigr]_{\vert X^{[n+1,n]}\be}\he$.
\end{definition}
The class $l$ determines the Chern classes of the relative incidence \sh\ in the following way:
\begin{proposition}\label{d}
 For all $i$, $n$  in $\N\ee$, $\mu _{i,\,n}\he=\sigma _{*}\he(l^{i})$.
\end{proposition}
\begin{proof}
 We recall well known facts from the classical theory (see \cite{SchHilDa}, \cite{SchHilCh}, \cite{SchHilTi}). If $X$ is a quasi-projective surface and $Y_{n}\he$ is the incidence locus in $\xn\tim X$, then: 
\begin{enumerate}
 \item [(i)] $\jj_{\yn}\he$ admits a locally free resolution of length $2$.
\item[(ii)] $\xna{n+1,n}\simeq\P\bigl( \jj_{\yn}\bigr)$ and $\xna{n+1,n}$ is smooth.
\item[(iii)] If $\sutrgd{\aaa }{\bbb }{\jj_{\yn}}$ is a resolution of $\jj_{\yn}$, 
$\apl{\pi }{\P(\bbb )}{\xn\tim X}$ is the projection and $s$ is the associated section of $\pi \ee\be\aaa \ee\be(1)$, then $s$ is transverse to the zero section.
\end{enumerate}
Property (iii) follows from (ii) since the vanishing locus of $s$ (with its schematic structure) is isomorphic to $\P\bigl(\jj_{\yn}\bigr)$. Note that we use here Grothendieck's convention for projective bundles. 
\par\medskip 
We adapt now these properties to the relative setting as follows. First we can suppose that 
$W$, $\tw$, $\jrel{n}$ and $\jrel{n\tim 1}$ satisfy the compatibility condition: $W\tim_{\snx}\he \bigl( \snx\tim X\bigr)\suq\tw$ as \rsas s. Let $\ti{\yg}_{n}\he$ be the relative singular analytic subspace of $\tw_{\re}^{[n]\tim[1]}$ defined by 
\[
\ti{\yg}_{n}\he\!=\!\bigl\{(\xi,w,\xb,p)\ \textrm{in}\ \tw_{\re}^{[n]\tim[1]}\ \textrm{such that}\ w\in\supp{\xi }\bigr\}\cdot 
\]
If we consider the following diagram 
\[
\xymatrix{
W^{[n]}_{\re}\tim_{\snx}\he W\ar[d]&\bigl( W^{[n]}_{\re}\tim_{\snx}\he W\bigr)\tim_{\snx}\he\bigl( \snx\tim X\bigr)\,\ar[d]\ar[l]\ar@{^{(}->}[r]&\tw^{[n]\tim[1]}_{\re}\ar[d]\\
\snx&\snx\tim X\ar[l]_-{\pr_{1}\he}\ar@{=}[r]&\snx\tim X
}
\]
$\xn\tim X$ embeds into the three relative smooth analytic spaces in a compatible way with the two morphisms on the first line. Therefore $\mu _{i,n}\he=c_{i}\he\bigl( \ore{\ti{\yg}_{n}}\bigr)_{\vert\xn\tim X}\he$. We will denote the two relative smooth analytic spaces $\tw^{[n+1,n]}_{\re}$ and $\tw^{[n]\tim[1]}_{\re}$ by $\xg$ and $\xgp$ respectively.
\par\medskip 
Let us take a family of charts $\bigl\{\apliso{\phi_{i}\he }{\Omega _{i}\he\tim V_{i}\he}{\tw_{i}\he}\bigr\}$ on $\tw$, where the  $\Omega _{i}\he$'s are (possibly non connected) open sets in $\C^{2}$. They provide relative holomorphic charts $\phi _{i}^{[n+1,n]}$ and $\phi _{i}^{[n]\tim[1]}$ of $\xg$ and $\xgp$. For each index $i$, we pick a locally free resolution $\sutrgd{\aaa _{i}\he}{\bbb _{i}\he}{\jj_{Y_{n,i}}\he}$
of length $2$ of $\jj_{Y_{n,i}}\he$, where $Y_{n,i}\he$ is the incidence locus in $\Omega _{i}^{[n]}\tim \Omega _{i}\he$. By the very construction of global smooth resolutions for relatively coherent \sv\ (Proposition \ref{NouvProp} (i)), there exists a global resolution 
$\sutrgd{\aaa }{\bbb }{\jj_{\ti{\yg}_{n}}^{\re,\infty }}$ of length $2$ of
$J_{\ti{\yg}_{n}}^{\re,\infty }\oti_{\ore{\xgp}}\he\ci_{\xgp}$ such that for all $i$
$\sutrgd{\aaa _{i}^{\,\infty }}{\bbb _{i}^{\,\infty }}{\jj_{\ti{\yg}_{n}}^{\re,\infty }}$ is a sub-resolution of the global one. Let $\apl{\pi }{\P(\bbb)}{\xgp}$ be the projective bundle of $\bbb$ and $s$ the section of $\pi \ee\be\aaa \ee\be\oti_{\ci_{\P(\bbb)}}\he\ci_{\P(\bbb)}(1)$ obtained by the morphism 
$\sutr{\pi \ee\be\aaa }{\pi \ee\be\bbb }{\ci_{\P(\bbb)}(1)}$. Then we have the following result:
\begin{lemma}\label{e}
 $\he$\par
\begin{enumerate}
 \item [(i)] The vanishing locus $\textrm{Z}(s)$ of $s$ is canonically isomorphic to $\xg$.
\item[(ii)] After performing base change from $\snx\tim X$ to $X^{n}\be\tim X$, the section $s$ is transverse to the zero section.
\item[(iii)] If $j$ is the embedding of $\xg$ into $\P(\bbb)$, then 
$\ore{\xg}(-D_{\re}\he)\oti_{\ore{\xg}}\he\ci_{\xg}\simeq j\ee\be\, \ci_{\P(\bbb)}(1)$.
\end{enumerate}
\end{lemma}
\begin{proof}
(i) We argue locally on $\xgp$. In relative coordinates $(z,v)$, the local resolutions of $\jj^{\re,\infty }_{\ti\yg_{n}}$ is of the form 
$\xymatrix@C=17pt{0\ar[r]&\T^{r}\be\ar[r]^-{M}&\T^{r+1}\be,}$
where $\T=\ci_{\xg}$ and $M$ is a $(r+1)\tim r$ matrix with holomorphic coefficients in the variable $z$. We can locally split the injection of $\aaa ^{\infty }_{i}$ in $\aaa $. The global resolution of $\jj^{\re,\infty }_{\ti\yg_{n}}$ is therefore locally isomorphic to 
$\xymatrix@C=17pt{0\ar[r]
&\T^{r}\be\oplus\T^{m}\be\ar[r]^-{\phi }&\T^{r+1}\be\oplus\T^{m}\be
}$ where $\phi =
\begin{pmatrix}
 M&0\\0&\id
\end{pmatrix}\cdot
$
Over a small open set of $\xgp$, $\P(\bbb)$ is the trivial projective bundle with fiber $\P^{r+m}\be(\C)\ee\be$ and $s$ is given in coordinates by 
$s(u,z,v)(\alpha ,\beta )=u\bigl( M(z)\alpha ,\beta \bigr)$, where $u\in\P^{r+m}\be(\C)\ee\be$ and $(\alpha ,\beta )\in\C^{r+m}$ are the coordinates in the fibers of $\pi \ee\be \aaa$. Therefore,
\[
\textrm{Z}(s)=\bigl\{(u,z,v)\ \textrm{such that}\ u_{\vert\C^{m}}\he=0\ \textrm{and for all}\ \alpha \in \C^{r}\be,\ u\bigl( M(z)\alpha \bigr)=0\bigr\}\cdot
\]
This implies that if we consider the local embedding $\xymatrix{\P(\bbb _{i}^{\infty} )\,\ar@^{{(}->}[r]&\P(\bbb)}$ over $\xgp$ given by the splitting of the local resolution in the global one, then $\textrm{Z}(s)$ lies in $\P(\bbb_{i}^{\infty })$. It can be seen straightforwardly that the embedding of $\textrm{Z}(s)$ into $\P(\bbb_{i}^{\infty })$ is independant of the splitting. Let $\apl{\ti\pi }{\P(\bbb_{i}\he)}{\Omega _{i}^{[n]}\tim\Omega _{i}\he}$ be the projective bundle of $\bbb_{i}\he$ and $\ti{s}$ the section of $\ti\pi \ee\be\aaa_{i}\ee(1)$ given by the morphism $\sutr{\ti{\pi }\ee\be\aaa_{i}\he}{\ti{\pi }\ee\be\bbb_{i}\he}{\oo_{\P(\bbb_{i}\he)}(1).}$ Then $\textrm{Z}(\ti{s})=\Omega _{i}^{[n+1,n]}$ and $\ti{s}$ is transverse to the zero section. This proves that over $\Omega _{i}^{[n]}\tim\Omega _{i}\he\tim V$,
$\textrm{Z}(s)=\textrm{Z}(\ti{s})\tim V=\Omega _{i}^{[n+1,n]}\tim V$ so that $\textrm{Z}(s)$ is abstractly isomorphic over $\tw_{i}^{[n]\tim[1]}$ to $\tw_{i}^{[n+1,n]}$.
\par\medskip 
(ii) In local coordinates, if for any $u\in\P^{r+m}\be(\C)\ee\be$ we define $u_{1}\he=u_{\vert\C^{r+1}}\he$ and $u_{2}=u_{\vert\C^{m}}\he$, then $s$ is given near its zero locus by $\apl{s(u,z,v)}{(\alpha ,\beta )}{\bigl( \ti{s}(u_{1}\he,z)(\alpha ),u_{2}\he(\beta )\bigr)}$. After base change, the variable $v$ lies in the smooth manifold $X^{n}\be\tim X$ and $s$ is clearly transverse to the zero section since $\ti{s}$ is.
\par\medskip 
(iii) We have a chain of morphisms $\xymatrix{\tw_{i}^{[n+1,n]}\,\ar@{^{(}->}[r]&\P(\bbb_{i}^{\,\infty })\,\ar@{^{(}.>}[r]&\P(\bbb)_{\vert \tw_{i}\he}\he}$ where the last morphism is only defined locally on $\P(\bbb^{\infty }_{i})$. To conclude, notice that $\ore{\P(\bbb_{i}^{\infty })}(1)$ is canonically isomorphic to $\ore{\xg}\bigl( -D_{\re}\bigr)_{\vert\tw_{i}\he}\he$ and that the restriction of $\ci_{\P(\bbb)_{\vert\tw_{i}}}$ to $\P(\bbb_{i}^{\,\infty })$ is canonically isomorphic to 
$\ci_{\P(\bbb_{i}^{\,\infty })}(1)$.
\end{proof}
Since $\P(\bbb)=\bigl[\P(\bbb)\tim_{\snx\tim X}\he(X^{n}\be \tim X)\bigr]/\sg_{n}\he$, Lemma \ref{e} (ii) implies that the homology class of the vanishing locus of $s$ in $H_{8n+8}\he(\P(\bbb),\Q)$ is Poincar\'{e} dual to the top Chern class of $\pi \ee\be\aaa\ee\be\oti_{\ci_{\P(\bbb)}}\he\ci_{\P(\bbb)}(1)$. 
Let us consider the following diagram:
\[
\xymatrix{
\P(\bbb)\ar[d]_-{\pi }&\xg\supseteq D_{\re}\he\ar[l]_(.54){j}\ar[dl]^{\sigma _{\re}\he}\\
\xgp
}
\]
If $\varepsilon =c_{1}\he\bigl( \ci_{\P(\bbb)}(1)\bigr)\in H^{2}\be(\P(\bbb),\Q)$, then by Lemma \ref{e} (iii), $l=j\ee\be\varepsilon _{\vert X^{[n+1,n]}\be}\he$. 
Now
\begin{align*}
 \sigma _{\!\re*}\he(j\ee\be\varepsilon ^{i}\be)&=\pi _{*}\he([\xg]\,.\,\varepsilon ^{i}\be)=\sum_{k=0}^{d}c_{k}\he(\aaa\ee\be)\,\pi _{*}\he(\varepsilon ^{d+i-k}\be)=\sum_{k=0}^{d}c_{k}\he(\aaa\ee\be)\,s_{i-k}\he(\bbb\ee\be)\\
&=c_{i}\he(\aaa\ee\be-\bbb\ee\be)=(-1)^{i}c_{i}\bigl( \jj_{\ti{\yg}_{n}\he}^{\re,\,\infty }\bigr)=(-1)^{i}c_{i}\he\bigl( \oo^{\re,\,\infty }_{\ti{\yg}_{n}\he}\bigr)\qquad\textrm{since}\ i\ge 1.
\end{align*}
We have used here the Gysin morphism $\sigma _{\!\re *}\he$ with \mbox{$\Q$-coefficients}, which is possible since $\xg$ and $\xgp$ are rationally smooth. To conclude, we consider the diagram
\[
\xymatrix@C=40pt@R=20pt{
\xg\ar[d]_-{\sigma _{\!\re}\he}&\,X^{[n+1,n]}\be\ar[d]^-{\sigma }\ar@{_{(}->}[l]\\
\xgp&\,\xn\tim X\ar@{_{(}->}[l]
}
\]
and get:
\[
\sigma _{*}\he\, l^{i}=\sigma _{*}\he(j\ee\be\varepsilon ^{i}\be)_{\vert \xna{n+1,n}}\he=\bigl[ \sigma _{\!\re*}\he(j\ee\be\varepsilon ^{i}\be)\bigr]_{\vert\xn\tim X}\he=(-1)^{i}c_{i}\bigl( \oo^{\re,\infty }_{\ti{\yg}_{n}\he}\bigr)_{\vert\xn\tim X}\he=(-1)^{i}\mu _{i,n}\he. 
\]
\end{proof}
\subsection{Computation of $TX^{[n]}$ in \mbox{$K$-theory}}\label{m}
Let $W$ be a \nbh\ of the incidence locus $Z_{n}\he$ in $\snx\tim X$ and $J^{\re}_{n}$ a relative integrable structure on it. 
\begin{definition}\label{g} 
 We define ${\kappa }_{n\vphantom{1}}\he$ as the class of the locally \mbox{$\ore{W^{[n]}_{\re}}$-free} \sh\ $T^{\re}\be W^{[n]}_{\re}$ in $\kre{}\bigl( W^{[n]}_{\re}\bigr)$. 
\end{definition}
\noindent The topological class of $\kappa _{n\vphantom{1}}\he$ can be described as follows:
\begin{lemma}\label{k}
 The restriction to $\xn$ of the topological class associated to $\kappa _{n\vphantom{1}}\he$ is the class of the complex vector bundle $T\xn$ in $K(\xn)$.
\end{lemma}
\begin{proof}
 If $J^{\re}_{\vphantom{!}n}$ satisfies the conditions $(\mathscr{C})$ listed in \cite{SchHilVo1} page 711, then $X^{[n]}_{J_{n}^{\re}}$ is smooth. Besides, the construction of the almost-complex structure done in \cite{SchHilVo1} shows that $T\xn$ and $T^{\re}\be W^{[n]}_{\re}{}_{\vert \xn}\he$ have the same class in $K(\xn)$. An arbitrary $J^{\re}_{\vphantom{!}n}$ can be joined by a smooth path $\bigl\{J_{n,\,t}^{\re}\bigr\}_{t\in[0,1]}\he$ to another relative integrable structure satisfying the conditions $(\mathscr{C})$. By rigi\-dity of the topological \mbox{$K$-theory}, the class of $T^{\re}\be\bigl[ W^{[n]}_{\re},J^{\re}_{n,\,t}\bigr]_{\vert \xn}\he $ in $K\bigl( \xn\bigr)$ is independent of $t$. This yields the result.
\end{proof}
\begin{remark}\label{h}
 For an arbitrary $J^{\re}_{\vphantom{!}n}$, $X^{[n]}_{J^{\re}_{n}}$ is only a topological manifold (see \cite{SchHilGri}). Therefore, the advantage of using $T^{\re}\be W^{[n]}_{\re}$ instead of $TX^{[n]}_{J_{n}^{\re}}$ is that this complex vector bundle is defined for \emph{any} relative integrable complex structure.
\end{remark}

\begin{proposition}\label{CinqNeuf}
 In $K^{\re}\be\bigl( W^{[n]}_{\re}\bigr)$, the following identity holds: 
\[
\kappa _{n\vphantom{1}}\he=p_{*}\he\bigl( \oo_{n}^{\re}+\oo_{n}^{\re\vee}-\oo_{n}^{\re}\, .\, \oo_{n}^{\re \vee}\bigr).
\]
\end{proposition}
\begin{proof}
We have $TW^{[n]}_{\re}=p_{*}\he\,\hoo\bigl( \jj_{n}^{\re},\ore{n}\bigr)$. Since $\supp\bigl( \ore{n}\bigr)$ has relative codimension $2$ in $W^{[n]}_{\re}\tim_{\snx}\he W$, $\ext^{i}\be\bigl( \ore{n},\ore{\vphantom{n}}\bigr)=0$ for $i<2$. Besides, $\jj_{n}^{\re}$ locally admits a free resolution of length $2$. Hence we get the following equalities in $\kre{\yg_{n}\he}\bigl( W^{[n]}_{\re}\tim_{\snx}\he W\bigr)$:
\begin{align*}
 \jj^{\re\vee}_{n}\,.\,\ore{n}&=\hoo\bigl( \jj_{n}^{\re},\ore{n}\bigr)-\ext^{1}\be\bigl( \jj_{n}^{\re},\ore{n}\bigr)+\ext^{2}\be\bigl( \jj_{n}^{\re},\ore{n}\bigr)\\
&=\hoo\bigl( \jj_{n}^{\re},\ore{n}\bigr)-\ext^{2}\be\bigl( \ore{n},\ore{n}\bigr)\\
&=\hoo\bigl( \jj_{n}^{\re},\ore{n}\bigr)-\ext^{2}\be\bigl( \ore{n},\ore{\vphantom{n}}\bigr)\\
&=\hoo\bigl( \jj_{n}^{\re},\ore{n}\bigr)-\oo_{n}^{\re\vee}
\end{align*}
so that $p_{*}\he\hoo\bigl( \jj_{n}^{\re},\ore{n}\bigr)=p_{*}\he\bigl[\bigl( \ore{\vphantom{n}}-\oo_{n}^{\re\vee}\bigr)\,.\,\ore{n}+\oo_{n}^{\re\vee} \bigr]$
\end{proof}
\subsection{Comparison of $TX^{[n]}$ and $TX^{[n+1]}$ via the incidence variety $X^{[n+1,n]}${}}\label{Comparaison}
Let ${W}$ be a \nbh\ of $Z_{n\tim 1}\he$ in $\snx\tim X\tim X$ and $J_{n\tim 1}^{\re }$ be a relative integrable complex structure on ${W}$. Let us consider the following morphisms over $\snx\tim X$:
\begin{enumerate}
 \item [$_{\ds *}$] $\apl{\rho _{\re}\he}{{W}_{\re}^{[n+1,\, n]}}{{W}}$ is the residual map.
\item [$_{\ds *}$] $\apl{j_{\rel}\he=\bigl( \id,\rho _{\re}\he\bigr)}{{W}_{\re}^{[n+1,\, n]}}{{W}_{\re}^{[n+1,\, n]}\tim_{\snx\tim X}\he{W}.}$
\item [$_{\ds *}$] $p$ is the first projection 
$\xymatrix@C=17pt{
{W}_{\re}^{[n+1,\, n]}\tim_{\snx\tim X}\he{W}\ar[r]&{{W}_{\re}^{[n+1,\, n]}}
.}$
\item [$_{\ds *}$] If $\apl{f}{\xg}{\xgp}$ is a morphism over $\snx\tim X$,
we define
\[
\apl{f_{{W}}\he:=\mkern -2 mu f\tim_{\snx\tim X}\he\id_{W}\he}{\xg\tim_{\snx\tim X}\he{W}}{\xg'\tim_{\snx\tim X}\he{W}.}
\]
\item [$_{\ds *}$] 
$\xymatrix{{\psi }:{W}_{\re}^{[n+1,\, n]}\ar[r]&{W}_{\re}^{[n+1]}}$ and
$\xymatrix{{\phi }:{W}_{\re}^{[n+1,\, n]}\ar[r]&{W}_{\re}^{[n]}}$ are the canonical projection maps.
\item [$_{\ds *}$] 
$\xymatrix{{\sigma_{\re}\he }=\bigl( {\phi },\rho_{\re}\he\bigr):W_{\re}^{[n+1,\, n]}\ar[r]&{W}_{\re}^{[n]}\tim_{\snx\tim X}{W}.}$
\end{enumerate}
We introduce the following relatively coherent sheaves:
\begin{enumerate}
\item [$_{\ds *}$] $\ti{\oo}_{n}^{\re}$ and $\ti{\oo}_{n+1}^{\re}$ are the relative incidence structure sheaves on ${W}_{\re}^{[n]}\tim_{\snx\tim X}\he W$ and ${W}_{\re}^{[n+1]}\tim_{\snx\tim X}\he W$.
\item [$_{\ds *}$] $\LL=\oo^{\re}\be(-D_{\re}\he)$, where $D_{\re}\he\suq{W}_{\re}^{[n+1,\, n]}$ is the relative exceptional divisor.
\item [$_{\ds *}$] $\Delta _{\re}\he$ is the relative diagonal in ${W}\tim_{\snx\tim X}\he {W}$ and $\oo^{\re}_{\Delta _{\re}}$ is the associated structure sheaf.
\end{enumerate}
Then we have the following properties which are immediate consequences of the same results in the integrable case:
\begin{lemma}\label{f}
 $\he$\par
\begin{enumerate}
\item [(i)] On ${W}_{\re}^{[n+1,n]}\tim_{\snx\tim X}\he{W}$ we have $j_{\re *}\he\LL =p\ee\be \LL\oti\rho \ee_{\re,W}\oo_{\Delta _{\rel}}^{\rel}$.\label{EqRef}
\item [(ii)] We have an exact sequence on ${W}_{\re}^{[n+1,\, n]}\tim_{\snx\tim X}\he{W}$\emph{:}
\[
\sutrgdpt{j_{\re*}\LL\he}{{\psi }_{W}\ee\, \ti\oo_{n+1}^{\re}}{{\phi }_{W}\ee\,\ti \oo_{n}^{\re}}{.}
\]
\item [(iii)] $\rho \ee_{\re,W}\oo_{\Delta _{\re}}^{\re}=\rho \pee  _{\re,W}\oo_{\Delta _{\re}}^{\re}$,\quad ${{\psi }_{W}\ee\,\ti \oo_{n+1}^{\re}}={{\psi }_{W}\pee  \,\ti \oo_{n+1}^{\re}}$\quad and \quad ${{\phi }_{W}\ee\,\ti \oo_{n}^{\re}}={{\phi }_{W}\pee  \,\ti \oo_{n}^{\re}.}$ 
\end{enumerate}
\end{lemma}
\par\bigskip 
\noindent We suppose that there exist a \nbh\ $\tw$ of $Z_{n+1}\he$ in $S^{n+1}\be X\tim X$ and a relative integrable complex structure $J_{n+1}^{\re}$ on it such that $W=\tw\tim_{S^{n+1}\be X}\he (\snx\tim X)$. This means that for all $(\xb,p)$ in $\snx\tim X$, 
$W_{\xb,\,p}\he=\tw_{\xb\, \cup\,  p}\he$ and
$J_{n\tim 1,\, \xb,\,p}^{\re}=J_{n+ 1,\, \xb\, \cup\, p}^{\re}$. Then we get two weak morphisms
\[
\apl{\ti\psi }{{W}_{\re}^{[n+1,\, n]}}{{W}_{\re}^{[n+1]}}\ \textrm{and}\quad\apl{\ti\psi_{W}\he }{{W}_{\re}^{[n+1,\, n]}\tim_{\snx\tim X}\he{W}}{{W}_{\re}^{[n+1]}\tim_{S^{n+1}\be X}\he{W}}
\]
(see Definition \ref{DefDeuxAppenCh4} (ii))
obtained by composing $\psi $ and $\psi _{W}\he$ with the base change map from $S^{n+1}\be X$ to $\snx\tim X$ given by $\flgdba{(\xb,p)}{\,\xb\,\cup\,p}$. Then ${\ti{\psi }_{W}\ee\,\ti \oo_{n+1}^{\re}}={{\psi }_{W}\ee\,\oo_{n+1}^{\re}}$.
Let $\ti{\kappa }_{n\vphantom{1}}\he$ be the class of $T^{\re}\be{W}^{[n]}_{\re}$ in 
$K^{\re}\be\bigl( {W}^{[n]}_{\re}\bigr)$.
\begin{proposition}\label{SansLabBis}
In $K^{\re}\be\bigl( \tw^{[n+1,\, n]}_{\re}\bigr)$ we have
\begin{align*}
 \ti\psi \pee \be \kappa _{n+1}\he={\phi }\pee \be\ti{\kappa } _{n\vphantom{1}}\he+\LL+\LL^{\vee}\be.\, \rho  \pee _{\re}
K_{W}^{\re\vee}-&\rho \pee _{\re}
\bigl( \oo_{W}^{\re}-T^{\re}\be W+K^{\re\vee}_{W}\bigr)\\&-\LL\, .\, {\sigma_{\re}\pe} \be\ti{\oo}_{n}^{\re\vee}-\LL^{\vee}\be.\, \rho \pee _{\re}K_{W}^{\re\vee}\, .\, {\sigma_{\re}\pee} \be\ti{\oo}_{n}^{\re}.
\end{align*}
\end{proposition}
\begin{proof}
Let $\apl{\ti{p}}{\tw_{\re}^{[n+1]}\tim_{S^{\, n+1}\be X}\he\tw}{\tw_{\re}^{[n+1]}}$ be the projection on the first factor.
 By Proposition \ref{CinqNeuf}, $\ti\psi \pee \be \kappa _{n+1}\he=\ti\psi \pee \be \ti p_{*}\he\bigl( \oo_{n+1 }^{\re}+\oo_{n+1 }^{\re\vee}-\oo_{n+1 }^{\re}\, .\, \oo_{n+1 }^{\re\vee}\bigr)$. Let us consider the cartesian diagram
\[\xymatrix@C=45pt{W_{\re}^{[n+1,\, n]}
\tim_{\snx\tim X}\he W\ar[r]^(.53){\ti\psi _{W}\he}\ar[d]_{{p}}&\tw_{\re}^{[n+1]}\tim_{S^{\, n+1}\be X}\he\tw\ar[d]^{\ti p}\\
W_{\re}^{[n+1,\, n]}\ar[r]^{\ti\psi }&\tw_{\re}^{[n+1]}
}\]
where the first column lies over $\snx\tim X $ and the second one over $S^{\, n+1}\be X$. Since $\ti p$ is finite on $\supp\bigl( \ore{n+1}\bigr)$, we obtain by Proposition \ref{PropUnBisAppendCh4} (iv) and Lemma \ref{f} (i) and (ii) the following relations in $\kre{}\bigl( W_{\re}^{[n+1,n]}\bigr)$:
\begin{align*}
 \ti\psi \pee \be \kappa _{n+1}\he&={p}_{*}\he\ti\psi \pee _{W}\bigl( \oo_{n+1}^{\re}+\oo_{n+1}^{\re\vee}-\oo_{n+1}^{\re}\, .\, \oo_{n+1}^{\re\vee}\bigr)
={p}_{*}\he{\phi}\pee _{W}\bigl( \ti\oo_{n}^{\re}+\ti\oo_{n}^{\re\vee}-\ti\oo_{n}^{\re}\, .\, \ti\oo_{n}^{\re\vee}\bigr)\\
&\hphantom{=}+{p}_{*}\he\Bigl[{p}\pee \be \LL\, .\, \rho \pee _{\re,W}\, \oo^{\re}_{\Delta _{\re}\he}+
{p}\pee \be \LL^{\vee}\be .\, \rho \pee _{\re,W}\, \oo^{\re\vee}_{\Delta_{\re}\he }-
{p}\pee \be \bigl( \LL\, .\, \LL^{\vee}\be\bigr)\, .\,  \rho \pee _{\re,W}\bigl( \oo^{\re}_{\Delta _{\re}\he}.\, \oo^{\re\vee}_{\Delta _{\re}\he}\bigr) \\
&\hphantom{=+\ti{p}_{*}\he\Bigl[{p}\pee \be \LL\, .\, \rho \pee _{\re,W}\, \oo^{\re}_{\Delta _{\re}\he}}-{p}\pee \be\LL\, .\, \rho \pee _{\re,W}\oo^{\re}_{\Delta _{\re}\he}\, .\, {\phi }_{W}\pee \, \ti\oo^{\re}_{n}-{p}\pee _{\be}\LL^{\vee}.\, \rho \pee _{\re,W}\, \oo^{\re\vee}_{\Delta _{\re}\he}\, .\,\, {\phi }_{W}\pee \, \ti{\oo}_{n}^{\re}
\Bigr].
\end{align*}
\par\medskip 
If $\apl{q}{W^{[n]}_{\re}\tim_{\snx\tim X}\he W}{W^{[n]}_{\re}}$ is the projection on the first factor, then by Proposition \ref{CinqNeuf}, 
$\ti{\kappa }_{n\vphantom{1}}\he=q_{*}\he\bigl( \ti\oo_{n}^{\re}+\ti\oo_{n}^{\re\vee}-\ti\oo_{n}^{\re}\, .\, \ti\oo_{n}^{\re\vee}\bigr)$ so that, since $q$ is finite on $\supp\bigl( \ore{n}\bigr)$, Proposition \ref{PropUnBisAppendCh4} (iv) yields:
${\phi }\pee \be\ti{\kappa } _{n\vphantom{1}}\he={p}_{*}\he{\phi }\pee _{W}\bigl( \ti\oo_{n}^{\re}+\ti\oo_{n}^{\re\vee}-\ti\oo_{n}^{\re}\, .\, \ti\oo_{n}^{\re\vee}\bigr)$. We also consider the diagram
\[\xymatrix@C=45pt{W_{\re}^{[n+1,\, n]}
\tim_{\snx\tim X}\he W\ar[r]^(.53){\rho _{\re,W}\he}\ar[d]_{{p}}&W\tim_{\snx\tim X}\he W\ar[d]^{r}\\
W_{\re}^{[n+1,\, n]}\ar[r]^{\rho _{\re}\he }&W
}\]where all the terms are over $\snx\tim X$. Since $r$ is injective on $\Delta _{\re}\he$, by Proposition \ref{PropUnBisAppendCh4} (iv) again,
\begin{align*}
 \ti\psi \pee \be \kappa _{n+1}\he&={\phi }\pee  \ti{\kappa}_{n\vphantom{1}}\he+\LL\, .\, \rho \pee _{\re}r_{*}\he\oo_{\Delta _{\re}\he}^{\re}+\LL^{\vee}\be.\, \rho \pee _{\re}r_{*}\he\oo_{\Delta _{\re}\he}^{\re\vee}-\rho \pee _{\re}r_{*}\he\bigl(\oo_{\Delta _{\re}\he}^{\re}\,.\,  \oo_{\Delta _{\re}\he}^{\re\vee} \bigr)\\
&-\LL\, .\, {p}_{*}\he\bigl( j_{\re*}\he\oo_{W_{\re}^{[n+1,\, n]}}^{\re}\, .\, {\phi }_{W}\pee \, \ti{\oo}_{n}^{\re\vee}\bigr)
-\LL^{\vee}\be.\, {p}_{*}\he\Bigl[ \bigl( \rho _{\re,W}\pee \oo^{\re}_{\Delta _{\re}\he}\bigr)^{\vee}\be \, .\, {\phi }\pee _{W}\, \ti{\oo}_{n}^{\re}\Bigr].
\end{align*}
Now $r_{*}\he\, \oo_{\Delta _{\re}\he}^{\re}=\oo^{\re}_{W}$, $r_{*}\he\, \oo_{\Delta _{\re}\he}^{\re\vee}=K_{W}^{\re\vee}$ and if $\apl{\delta }{W}{W\tim_{\snx\tim X}\he W}$ is the diagonal injection, $\rho \pee _{\re,W}\, \oo_{\Delta _{\re}\he}^{\re\vee}=\rho \pee _{\re,W}\, \delta _{*}\he\, K_{W}^{\re\vee}=j_{\re*}\he\rho \pee _{\re}\, K_{W}^{\re\vee}$, thanks to Proposition \ref{PropUnBisAppendCh4} (v) and to the diagram
\[
\xymatrix@C=45pt{
W_{\re}^{[n+1,\, n]}\ar[r]^(.4){j_{\re}\he}\ar[d]^{\rho _{\re}\he}&W_{\re}^{[n+1,\, n]}\tim_{\snx\tim X}\he W\ar[d]^{\rho _{\re,W}\he}\\
W\ar[r]^(0.4){\delta }&W\tim_{\snx\tim X}\he W}
\]
\end{proof}
\subsection{Cohomological computations}\label{qwe}
We are now going to consider the identity we obtained in
Section \ref{Comparaison} from a cohomological point of view.
Recall that the classes $\mu _{i,n}\he$ in $H^{2i}\be\bigl( \xn\tim X,\Q\bigr)$ are defined by $\mu _{i,n}\he=c_{i}\he\bigl(\oo_{n}^{\, \rel}\bigr)_{\vert\xn\tim X}\he$. 
Let us consider the following maps:
\[
\xymatrix@R=8pt{
&X&\\
&&\\
&\xna{n+1,\, n}\ar[uu]_-{\rho }\ar[dl]_-{\lambda }\ar[dr]^-{\nu }\ar[dd]^-{\sigma }\\
\xn&&\xna{n+1}\\
&\xn\tim X&
}
\]
\begin{lemma}\label{EtToc}
 $\he$\par
\begin{enumerate}
 \item [(i)] $c_{i}\he\bigl( {\phi }\pee \be\ti{\kappa}_{n\vphantom{1}}\he\bigr)=\lambda \ee\be c_{i}\he\bigl( \xn\bigr)$.
\item[(ii)] $c_{i}\he\bigl( {\phi }_{W}\pee \, \ti{\oo}_{n}^{\re}\bigr)=(\lambda ,\id)\ee\be \mu _{i,n}\he$.
\item[(iii)] $c_{i}\he\bigl( {\sigma }\pee_{\re}\, \ti{\oo}_{n}^{\re}\bigr)_{\vert \, \xna{n+1,\, n}}\he=\sigma \ee\be \mu _{i,n }\he$.
\item[(iv)] $c_{i}\he\bigl( \rho _{\re}\ee\, T^{\re}\be W\bigr)_{\vert\, \xna{n+1,\, n}}=\rho \ee\be c_{i}\he(X)$.
\item[(v)] $c_{i}\he\bigl( \rho \pee _{\re,W}\, \oo_{\Delta _{\re}\he}^{\re}\bigr)_{\vert\, \xna{n+1,\, n}\tim X}\he=(\rho  ,\id)\ee c_{i}\he\bigl( \ci_{\Delta _{X}\he}\bigr)$.
\end{enumerate}

\end{lemma}
\begin{proof}
By Lemma \ref{Compl} (ii), the classes ${\phi }\pee \be\ti{\kappa}_{n\vphantom{1}}\he$, $ {\phi }_{W}\pee \, \ti{\oo}_{n}^{\re}$ and ${\sigma }\pee_{\re}\, \ti{\oo}_{n}^{\re}$ in topological \mbox{$K$-theory} are independent of the relative integrable structure $J^{\re}_{n\tim 1}$. Therefore, we can suppose that there exist a \nbh\ $\breve{W}$ of $Z_{n}\he$ in $\snx\tim X$ and a relative integrable structure $J_{n}^{\re}$ on it such that $\breve{W}\tim_{\snx}\he (\snx\tim X)\suq W$. This means that for $(\xb,p)$ in $\snx\tim X$, $\breve{W}_{\xb}\he\suq W_{\xb,\,p}\he$ and $J_{n\tim 1,\,\xb,\,p}^{\re}{}_{\vert \breve{W}_{\xb}\he}\he=J^{\re}_{n,\,\xb}$. Then (i), (ii) and (iii) are straightforward.
\par
For (iv), let $\widehat{W}=W\tim_{\snx}\he X$, where the base change morphism is given by the diagonal injection of $X$ in $\snx$. We consider the following diagram: 
\[
\xymatrix@C=45pt{W_{\re}^{[n+1,\, n]}\ar[r]^-{\rho _{\rel}\he}&W&\,\widehat{W}\ar@{_{(}->}[l]\\
\xna{n+1,\, n}\ar@{^{(}->}[u]\ar[r]_-{\rho }&X\he\ar@{^{(}->}[u]\ar@{=}[r]&X\he\ar@{^{(}->}[u]
}
\]
Then, $c_{i}\he\bigl( \rho _{\re}\ee\, T^{\re}\be W\bigr)_{\vert\, \xna{n+1,\, n}}=\rho  \ee\be c_{i}\he\bigl( T^{\re}\be\widehat{W}\bigr)_{\vert X}\he$. Since $\widehat{W}$ is a \nbh\ of 
$\Delta _{X}\he$ in $X\tim X$, 
$T^{\re}\widehat{W}_{\vert X}\he\simeq TX$, so that 
$c_{i}\he\bigl( T^{\re}\be\widehat{W}\bigr)_{\vert X}\he=c_{i}\he(X)$.
\par\medskip 
For (v), we extend the structure $J_{n\tim 1}^{\, \rel}$ to a relative integrable structure
$J_{n\tim 1\tim 1}^{\, \rel}$ in a \nbh\ $\ba W$ of 
$Z_{n\tim 1\tim 1}\he$ such that for any $(\xb,p,q)\in \snx\tim X\tim X $ near the incidence locus $p=q$, the equality $J_{n\tim 1\tim 1,\,\xb,\,p,\,q}^{\, \rel}=J_{n\tim 1,\,\xb,\,p}^{\, \rel}$ holds. 
This means that $W\tim_{\snx\tim X}\he(\snx\tim X\tim X)\suq\ba{W}$. We define $\widehat{W}$ by $\widehat{W}=\ba{W}\tim_{\snx\tim X\tim X}\he (X\tim X)$, the base change morphism from $\snx\tim X\tim X$ to $X\tim X$ being given by $\flgdba{(x,x')}{(\delta (x), x,x')}$, where $\apl{\delta }{X}{\snx}$ is the diagonal injection.
Then we consider the diagram:
\[
\xymatrix@C=50pt@R=25pt{
W^{[n+1,n]}_{\re}\tim_{\snx\tim X}\he W\ar[r]^-{\rho _{\rel,\,W}\he}\ar@{^{(}->}[d]&W\tim_{\snx\tim X}\he W\ar@{^{(}->}[d]\\
\ba W^{[n+1,n]}_{\re}\tim_{\snx\tim X\tim X}\ba{W}\ar[r]^-{\rho _{\rel,\,\ba{W}}\he}&\ba W\tim_{\snx\tim X\tim X}\he \ba W&\,\widehat{W}\tim_{X\tim X}\he\widehat{W}\ar@{_{(}->}[l]\\
\xna{n+1,n}\tim X\he\ar[r]_-{(\rho ,\,\id)}\ar@{^{(}->}[u]&X\tim X\he\ar@{^{(}->}[u]&X\tim X\he\ar@{=}[l]\ar@{^{(}->}[u]_(.45){\iota}
}
\]
If $\Delta _{\widehat{W}}^{\re}$ is the relative diagonal in $\widehat{W}\tim_{X\tim X}\he\widehat{W}$, then $\iota \ee\be \bigl[ \ci_{\Delta _{\widehat{W}}^{\re}}\bigr]=\bigl[ \ci_{\Delta _{X}\he}\bigr]$ in $K(X\tim X)$. This gives the result.
\end{proof}
\begin{remark}
Let $d_{i}\he$ be the \mbox{$i$-th} Chern class of $\ci_{\Delta _{X}\he}$ in $H^{2i}\be(X\tim X,\Q)$. By \cite{SchHilAtHi}, $d_{i}$ is a universal polynomial in $c_{1}\he(X)$ and $c_{2}\he(X)$ with rational coefficients. 
\end{remark}
\begin{proposition}\label{LaDer}
 $\he$\par
\emph{(i)} For all $i,n$ in $\N\ee\be$, $(\nu ,\id)\ee\be\mu _{i,n+1}\he-(\lambda ,\id)\ee\be\mu _{i,n}\he=\ds\sum_{k=0}^{i}\pr\ee_{1}\, l^{\, k}\be\, .\, (\rho ,\id)\ee\be\, d_{i-k}$, where 
$\mu _{i,n}\he$ and $l$ are defined in \ref{c} and \ref{l}. 
\par\medskip 
\emph{(ii)} For all $i,n$ in $\N\ee\be$, $\nu \ee\be c_{i}\he\bigl( \xna{n+1}\bigr)-\lambda \ee\be c_{i}\he\bigl( \xn\bigr)$ is a universal polynomial in the classes $l$, $\rho \ee\be c_{i}\he(X)$ and $\sigma \ee\be\mu _{i,n}\he$.
\end{proposition}
\begin{proof}
 This is a consequence of Lemma \ref{EtToc}, Lemma \ref{f} (i) and (ii), and of Proposition \ref{SansLabBis}.
\end{proof}
We are now going to perform the induction step.
\begin{proposition}\label{11}
If $n,m$ are positive integers, let $P$ be a polynomial in 
$
\pr_{0}\ee c_{i}\he\bigl( \xna{n+1}\bigr)$, $\pr\ee_{0k}\, \mu _{i,n+1}\he$, $\pr\ee_{kl}\, d_{i}\he$, $\pr_{k}\ee c_{i\he}(X)$ $ (1\le k,l\le m)$, which are cohomology classes on $\xna{n+1}\tim X^{\vphantom{[}m}\be.
$
Then there exists a polynomial $\ti{P}$ depending only on $P$, in the classes analogously defined on $\xn\tim X^{\vphantom{[}m+1}\be$, such that 
$\ds\int_{\xna{n+1}\tim X^{\vphantom{[}m}\be}P=\int_{\xn\tim X^{\vphantom{[}m+1}\be}\ti{P}$.
\end{proposition}
\begin{proof}
 We consider the incidence diagram
\par
\begin{pspicture}(-3.5,-3)(10,1)
\psset{xunit=0.2ex,yunit=0.2ex}
\psset{linewidth=0.1ex}
\rput(100,10){$\xna{n+1,n}\tim X^{m}\be$}
\rput(15,-40){$\xna{n+1}\tim X^{m}\be$}
\rput(185,-40){$\bigl( \xn\tim X\bigr)\tim X^{m}\be$}
\psline{->}(62,3)(25,-30)
\psline{->}(137,3)(174,-30)
\rput(35,-7){$\scriptstyle{(\nu ,\id)}$}
\rput(165,-7){$\scriptstyle{(\sigma ,\id)}$}
\end{pspicture}
\par
Since $(\nu ,\id)$ and $(\sigma ,\id)$ are generically finite of degrees $n+1$ and $1$, 
\[
\ds\int _{\xna{n+1}\tim X^{m}\be}P=\frac{1}{n+1}\int_{\xn\tim X^{m+1}\be}(\sigma ,\id)_{*}\he(\nu ,\id)\ee\be P.
\] 
By Proposition \ref{LaDer} (ii), 
$(\nu ,\id)\ee\be\, pr_{0}\ee\, c_{i}\he\bigl( \xna{n+1}\bigr)-(\sigma, \id)\ee\, pr_{0}\ee\, c_{i}\he\bigl( \xn\bigr)$ is a polynomial in the classes $pr_{0}\ee\, l$, $(\sigma ,\id)\ee\be \, pr_{1}\ee \, c_{i}\he(X)$ and $(\sigma,\id)\ee\be \, pr_{01}\ee\, \mu _{i,n}\he$. 
\par 
By Proposition \ref{d}, $(\sigma ,\id)_{*}\he\,  l^{\, i}=(-1)^{i}\be\, pr_{01}\ee\, \mu _{i,n}\he$. Thus, $(\sigma ,\id)_{*}\he\, (\nu ,\id)\ee\be\, pr_{0}\ee\, c_{i}\he\bigl( \xna{n+1}\bigr)$ is a polynomial in 
$pr_{01}\ee\, \mu _{i,n}\he$ and $pr_{1}\ee\, c_{i}\he(X)$. 
\par 
By Proposition \ref{LaDer} (i), 
$(\nu ,\id)\ee\be\, pr_{0k}\ee\, \mu _{i,n+1}\he-(\sigma ,\id)\ee\be\, pr_{0,k+1}\ee\, \mu _{i,n}\he$ is a polynomial in the classes $pr_{0}\ee\, l$ and $(\sigma \id)\ee\be\, pr_{1k}\ee\, d_{i}\he$. Then we can use Proposition \ref{d} again. 
\par 
To conclude, we use the relations 
$(\nu,\id)\ee\be\, pr_{kl}\he\, d_{i}\he=(\sigma ,\id)\ee\be pr_{k+1,\, l+1}\he\, d_{i}\he$ and $(\nu ,\id)\ee\be\, pr_{k}\ee\, c_{i}\he(X)=(\sigma ,\id)\ee\be\,  pr_{k+1}\ee\, c_{i}\he(X)$. 
\end{proof}
We can now finish the proof of Theorem \ref{QUATRE}. We write 
\[
\int_{\xn}P\Bigl( c_{1}\he\bigl( \xn\bigr),\dots,c_{2n}\bigl( \xn\bigr)\Bigr)=\int_{\xna{n-1}\tim X}\ti{P}_{1}\he=\int_{\xna{n-2}\tim X^{2}\be}\ti{P}_{2}\he
=\dots=\int_{X^{n}\be}\ti{P}
\] where $\ti{P}$ is a polynomial in the classes $pr_{k}\ee\, c_{i}\he(X)$ and $pr_{kl}\ee\, d_{i}\he$. Since $d_{i}\he$ is a polynomial in $c_{1}\he(X)$ and $c_{2}(X)$, we are done.
\begin{flushright}
 $\square$
\end{flushright}
\section{Appendix}\label{AppendCh4}
Our aim in this appendix is to develop a general framework for \sv\ on spaces which are fibered in smooth analytic spaces over a differentiable orbifold. This formalism is somehow heavy but necessary to carry out the computations of \cite{SchHilEGL} in a relative setting.
\subsection{Relative analytic spaces}\label{n}
Let $M$ be a smooth manifold and let $G$ be a finite subgroup of $\textrm{Diff}(M)$ (in fact, any effective differentiable orbifold would do, but no such generality is required here). Let $B=M/G$. If $\apl{p}{M}{B}$ is the projection, $B$ will be endowed with the \sh\ of rings 
$\ci_{B}:=p_{*}\he\bigl( \ci_{M}\bigr)^{G}\be\!\!.$ 
\begin{definition}\label{DefUnAppenCh4}
$\he$
\begin{enumerate}
 \item [(i)] A \emph{\rsas\ over $B$} is a ringed topological space $\bigl( \xg,\ore{\xg}\bigr)$ endowed with a surjective morphism $\apl{\pi }{\bigl( \xg,\ore{\xg}\bigr)}{\bigl( B,\ci_{B}\bigr)}$ such that $\bigl( \xg,\ore{\xg}\bigr)$ is locally on $\xg$ isomorphic over $B$ to a ringed topological space of the type $\bigl( Z\tim V,\ore{Z\tim V}\bigr)$, where $V$ is an open subset of $B$, $Z$ is a smooth analytic space and $\ore{Z\tim V}$ is the \sh\ of differentiable functions on $Z\tim V$ which are holomorphic on the slices $Z\tim\{v\}$, $v\in V$.
\item[(ii)] Let $\bigl( \xg,\ore{\xg}\bigr)$ be a \rsas\ over $B$ and let $\zg$ be a subset of $\xg$. We say that $\zg$ is \emph{a relative analytic subspace of $B$} if for all $x\in\xg$ and for all relative holomorphic chart $\apliso{\phi }{Z\tim V}{\,U_{x}\he}$, $\phi ^{-1}\bigl( \zg_{\vert U_{x}}\he\bigr)=Z'\tim V$, where $Z'$ is a (possibly singular) analytic subset of $Z$.
\end{enumerate}
\end{definition}
\begin{remark}\label{RemUnAppendCh4}
$\he$\par
-- If $\bigl( \xg,\ore{\xg}\bigr)$ is a \rsas\ over $B$, the fibers $\bigl(\xg_{b}\he\bigr)_{b\in B}\he$ are smooth analytic sets, but they do not form in general a fibration over $B$, since the projection map $\pi $ is not proper. However, they locally vary in a trivial way on $\xg$.
\par
-- The main example of a relative smooth analytic space we can keep in mind is $W^{[n]}_{\re}$ over $\snx$ where $W$ is a \nbh\ of the incidence set in $\snx\tim X$, endowed with a relative integrable structure parametrized by $\snx$.
\end{remark}
We now introduce morphisms between relative smooth analytic spaces.
\begin{definition}\label{DefDeuxAppenCh4}
 $\he$
\begin{enumerate}
 \item [(i)] Let $\xg$ and $\xgp $ be two \rsas s over $B$ and $\apl{f}{\xg}{\xgp }$ be a continuous map over $B$. We say that $f$ is a \emph{morphism} if for all $x\in \xg$ we can find trivializations of $\xg$ and $\xgp $ around $x$ and $f(x)$  in which $\apl{f}{Z\tim V}{Z'\tim V}$ has the form $\flgdba{(z,v)}{(g(z),v)}$, where $\apl{g}{Z}{Z'}$ is holomorphic.
\item[(ii)] Let $\xg$ and $\xgp $ be two \rsas s over $B$ and $B'$ and $\apl{f}{\xg}{\xgp }$ be a continuous map. We say that $f$ is a \emph{weak morphism} if there exist a smooth map 
$\apl{u}{B}{B'}$ and a morphism $\apl{\ti{f}}{\xg}{\xgp \tim_{B'}\he B}$ over $B$ such that $f$ is obtained by composing $\ti{f}$ with the base change morphism 
$\apl{u}{\xgp \tim_{B'}\he B}{\xgp .}$
\end{enumerate}
\end{definition}
\begin{remark}\label{RemDeuxAppendCh4}
 $\he$\par
-- If $\xg$ and $\xgp $ are two \rsas s over $B$ and $B'$, a continuous map $\apl{f}{\xg}{\xgp }$ over a smooth map $\apl{u}{B}{B'}$ which is holomorphic in the fibers is not a weak morphism in general.
\par
-- Weak morphisms can be characterized by their expression in suitable charts as it is the case for morphims. Indeed, $\apl{f}{\xg}{\xgp }$ is a weak morphism if for all $x\in\xg$ there exist trivializations of $\xg$ and $\xgp $ around $x$ and $f(x)$ in which 
$\apl{f}{Z\tim V}{Z'\tim V'}$ has the form $\flgdba{(z,v)}(\!g(z),u(v))$, where $\apl{g}{Z}{Z'}$ is holomorphic and $\apl{u}{V}{V'}$ is smooth.
\par
-- The term ``morphism'' has to be carefully understood because morphisms cannot \emph{a priori} be composed in this context. 
\end{remark}
\subsection{Relatively coherent sheaves}\label{q}
We introduce now the main object of our study.
\begin{definition}\label{DefQuatreAppendCh4}
 Let $\xg$ be a \rsas\ over $B$. A \sh\ $\ff$ of $\ore{\xg}$-modules will be \emph{\rc\ } if for all $x\in\xg$, if $\apliso{\phi }{Z\tim V}{\,U_{x}\he}$ is a trivialization of $\xg$ in a \nbh\ $U_{x}\he$ of $x$, there exists a coherent \sh\ $\baf$ on $Z$ such that 
\[
\phi ^{-1}\ff_{\vert U_{x}\he}\he\simeq \pru\baf\oti_{\pru\oz }\ore{Z\tim V}
\] 
as \sv\ of $\ore{Z\tim V}$-modules.
\end{definition}
\begin{remark}\label{RemQuatreAppenCh4}
 $\he$\par 
-- If $\xg$ is a \rsas\ over $B$ and $\zg$ is a relative analytic subspace of $\xg$, then 
$\ore{\zg}$ is a relatively coherent \sh\ on $\xg$.
\par
-- A \rc\  \sh\ on $\xg$ can also be defined by glueing conditions: it is given by a family of relative holomorphic charts $\bigl\{\apliso{\phi _{i}}{Z_{i}\he\tim V_{i}\he}{U_{i}\he}\bigr\}_{i}\he$ with transition functions
$\xymatrix@C=17pt{\phi _{ji}:Z_{ij}\he\tim V_{ij}\he\ar[r]^-{\sim}& Z_{ji}\he\tim V_{ji}\he,}$ and a family of coherent \sv\ $\bigl\{\baf_{i}\he\bigr\}_{i}\he$ on $\{Z_{i}\he\}$ together with isomorphisms
\[
\phi _{ji}^{-1}\Bigl[\bigl(\,\baf_{j}\he\oti_{\pru\oo_{Z_{j}}}\ore{Z_{j}\he\tim V_{j}\he}\bigr)_{\vert Z_{ji}\he\tim U_{ji}\he}\he\,\Bigr]\simeq\bigl(\,\baf_{i\he}\oti_{\pru\oo_{Z_{i}}}\ore{Z_{i}\he\tim V_{i}\he}\bigr)_{\vert Z_{ij}\he\tim U_{ij}\he}\he
\]
of \sv\ of $\ore{Z_{ij}\he\tim V_{ij}\he}$-modules satisfying the usual cocycle condition.
\par
-- Let $\ff$ be a \rc\  \sh\ on $\xg$ given by a family of \sv\ $\{\baf_{i}\he\}_{i\in I}\he$. Then, for any $b\in B$, if 
$J=\{i\in I\ \textrm{such that}\ b\in V_{i}\he\}$, the \sv\ $\{\baf_{i}\he\}_{i\in J}\he$ on 
$\{Z_{i}\he\tim b\}_{i\in J}\he$ patch together into a coherent \sh\ on $\xg_{b}\he$, which we will denote by $\ff_{b}\he$.
\end{remark}
The glueing procedure allows to perform standard operations for coherent \sv\ in the relative setting.
\begin{definition}\label{DefSixAppendCh4}
 $\he$
\begin{enumerate}
 \item [(i)] Let $\ff$ and $\g$ be \rc\  \sv\ on $\xg$ given by two families 
$\{\baf_{i}\he\}$ and $\{\bag_{i}\he\}$. The \sv\ $\ext^{\,k}\be (\ff,\g)$  and 
$\tor^{k}\be(\ff,\g)$ are defined by the families 
$\bigl\{\ext^{\,k}_{\oo_{Z_{i}}}(\baf_{i}\he,\bag_{i}\he)\bigr\}$ and
$\bigl\{\tor^{k}_{\oo_{Z_{i}}}(\baf_{i}\he,\bag_{i}\he)\bigr\}.$  We put 
$\hh om(\ff,\g)=\ext^{\,0}\be(\ff,\g)$ and $\ff\oti\g=\tor^{0}\be(\ff,\g)$.
\item[(ii)] Let $\apl{f}{\bigl( \xg,\ore{\xg}\bigr)}{\bigl( \xgp,\ore{\xgp}\bigr)}$ be a weak morphism between relative smooth analytic spaces. Let 
$\{\phi _{i}\he\}_{i\in I}\he$ and $\{\phi '_{i}\}_{j\in J}\he$ be two families of charts such that:
\begin{enumerate}
\item[--] $I\suq J$ and $f(U_{i}\he)\suq U_{i}'$.
\item[--] In the charts $\phi _{i}\he$ and $\phi _{i}'$, $f$ has the form 
$\flgdba{(z,v)}{(g_{i}\he(z),u_{i}\he(v)).}$
\end{enumerate}
If $\g$ is a \rc\  \sh\ on $\xgp$ given by a family $\{\bag_{j}\he\}_{j\in J}\he$, the \sv\ 
$\tor^{k}(\g,f)$ are defined by the families
$\Bigl\{\tor^{k}_{g^{-1}_{i}\oo_{Z_{i}'}}\bigl(g_{i}^{-1}\bag_{i}\he, \oo_{Z_{i}\he}\bigr)\Bigr\}_{i\in I}\he$. We put $f\ee\be\g=\tor^{0\be}(\g,f)$.
\item[(iii)] Let $\apl{f}{\bigl( \xg,\ore{\xg}\bigr)}{\bigl( \xgp,\ore{\xgp}\bigr)}$ be a morphism of \rsas s over $B$ and $\ff$ a \rc\ \sh\ on $\xg$ such that for all $b\in B$, $f_{b}\he$ is finite on $\supp(\ff_{b}\he)$. Let $\{\phi _{i}\he,U_{i}\he\}_{i\in I}\he$ and $\{\phi '_{j},U'_{j}\}_{j\in J}\he$ be two families of charts such that 
\begin{enumerate}
\item[--] $J\suq I$ and $f^{-1}(U'_{j})\cap\supp(\ff)\suq U_{j}$.
\item[--] If $W_{j}\he$ is a \nbh\ of $f^{-1}(U'_{j})\cap\supp(\ff)$, $\apl{f}{W_{j}\he}{U_{j}\he}$ has the form $\flgdba{(z,w)}{(g_{j}(z),v)}$.
\end{enumerate}
If $\ff$ is given by a family $\{\ff_{i\he}\}_{i\in I}\he$, the \sh\ $f_{*}\he\ff$ is defined by the family $\{g_{j*}\he\ff_{j}\he\}_{j\in I}\he$.
\end{enumerate}
\end{definition}
\begin{remark}\label{RemSixAppenCh4}
$\he$\par
-- For all $b$ in $ B$,
 \begin{align*}
 \ext^{k}\be (\ff,\g)_{b}\he&=\ext^{k}_{\oo_{\xg_{b}}}(\ff_{b}\he,\g_{b}\he),&
\tor^{k}\be (\ff,\g)_{b}\he&=\tor^{k}_{\oo_{\xg_{b}}}(\ff_{b}\he,\g_{b}\he),\\
\tor^{k}\be (\g,f)_{b}\he&=\tor^{k}_{f^{-1}_{b}\oo_{\xgp_{u(b)}}}\bigl( f^{-1}_{b}\g_{u(b)}\he,\oo_{\xg_{b}}\bigr)&
\bigl( f_{*}\he\ff\bigr)_{b}\he&=\bigl( f_{b}\he\bigr)_{*}\he\ff_{b}.
\end{align*}
\par
-- The finiteness hypothesis will be verified for the computations of Sections \ref{m} and \ref{Comparaison}. Therefore, we need not construct direct images in full generality.
\par
-- It is straightforward that 
$\hh om (\ff,\g)=\hh om_{\ore{\xg}}\he(\ff,\g),$\quad $\ff\oti\g=\ff\oti_{\ore{\xg}}\he\g,$\quad $
f\ee\be\g=f^{-1}\g\oti_{f^{-1}\ore{\xgp}}\he\ore{\xg}$ and 
that $f_{*}\he\ff$ is the usual direct image of $\ff$ by $f$ via the morphism 
$\flgd{\ore{\xgp}}{f_{*}\he\ore{\xg}}$. The associated derived results are also true and are consequences of Lemma 
\ref{LemUnAppenCh4}, (i). 
Thanks to the finiteness hypothesis, there are no higher direct images $R^{i}\be f_{*}\he$.
\end{remark}
\subsection{Analytic \mbox{$K$-theory} for relatively coherent sheaves}\label{p}
We are now going to define morphisms of \rc\  sheaves. The natural idea is to consider the \rc\  \sv\ on $\bigl( \xg,\ore{\xg}\bigr)$ as a full subcategory of $\textrm{Mod}\bigl( \ore{\xg}\bigr)$. It is not appropriate because this category would be non-abelian. Before giving the definition, we start with a preliminary flatness lemma which will be essential in the sequel. 
\begin{lemma}\label{LemUnAppenCh4}
 Let $U$ be an open set of $\R^{n}\be$, $G$ a finite group acting smoothly on $U$, $V=U/G$ and let $Z$ be a smooth analytic set. Then
\begin{enumerate}
 \item [(i)] $\ore{Z\tim V}$ is flat over $\pru\oz $,
\item [(ii)] $\ci_{Z\tim V}$ is flat over $\pru\oz $.
\end{enumerate}
\end{lemma}
\begin{proof}
 (i) Let $\apl{\delta }{U}{V}$ be the projection and $\mm$ be a \sh\ of $\pru\oz $-modules. Then
\[
(\delta ,\id)^{-1}\bigl( \mm\oti_{\pru\oz }\ci_{Z\tim V}\bigr)=
(\delta ,\id)^{-1}\mm\oti_{\pru\oz  }\bigl(\ci_{Z\tim U}\bigr)^{G}\be
\simeq\bigl[(\delta ,\id)^{-1}\mm\oti_{\pru\oz  }\ci_{Z\tim U}\bigr]^{G}\be.
\]
Since the functor $\flgdba{\ff}{\ff^{G}\be}$ from\ \ $\textrm{Mod}_{G}\he\bigl( \ci_{Z\tim U}\bigr)$\ \  to\ \ $\textrm{Mod}\bigl[ \bigl( \ci_{Z\tim U}\bigr)^{G}\be\,\bigr]$ is exact, it suffices to prove that $\ci_{Z\tim U}$ is smooth over $\pru\oz $. If $Y$ is a real-analytic manifold, we will denote by $\cio_{Y}$ the \sh\ of real-analytic functions on $Y$. Then, $\cio_{Z}$ is flat over $\oz$, $\cio_{Z\tim U}$ is flat over $\pru\cio_{Z}$ and $\ci_{Z\tim U}$ is flat over $\cio_{Z\tim U}$ by \cite[Th 2]{SchHilMa}.
\par\medskip 
(ii) As in (i), it suffices to prove that $\ore{Z\tim U}$ is flat over $\pru\oz$. Let 
$k=\lfloor (n+1)/2\rfloor$. Then $U\tim \R^{k}\be$ can be seen as an open subset $\ti{U}$ in $\C^{(n+k)/2}\be$. By \cite[Th 2 bis]{SchHilMa}, $\ore{Z\tim\ti{U}}$ is flat over $\oo_{Z\tim\ti{U}}\he$ and $\oo_{Z\tim\ti{U}}\he$ is flat over $\pru\oz$. Thus $\ore{Z\tim\ti{Y}} $ is flat over $\pru\oz$. If $\apl{q}{Z\tim\ti{U}}{Z\tim U}$ is the projection, $q^{-1}\ore{Z\tim U}$ is a direct factor of $\ore{Z\tim\ti{U}}$ (as 
$\pru\oz$-modules), so that $\ore{Z\tim U}$ is flat over $\pru\oz$.
\end{proof}
This being done, the definition of a morphism of \rc\ \sv\ runs as follows:
\begin{definition}\label{DefNeufAppendCh4}
Let $\ff$ and $\g$ be \rc\ \sv\ on a \rsas\ $\bigl( \xg,\ore{\xg}\bigr)$.
A morphism $u\in\textrm{Hom}_{\ore{\xg}}\he(\ff,\g)$ will be said to be \emph{strict} if for every $x\in\xg$ and any trivialization $\apliso{\phi }{Z\tim V}{\,U_{x}\he}$ in a \nbh\ of $x$, there exist two
isomorphisms $\phi ^{-1}\ff_{\vert U_{x}\he}\simeq \pru\baf\oti_{\pru\oz }\ore{Z\tim V}$ and $\phi ^{-1}\g_{\vert U_{x}\he}\simeq \pru\bag\oti_{\pru\oz }\ore{Z\tim V}$, where $\baf$ and $\bag$ are coherent on $Z$, and $v$ in $ \textrm{Hom}_{\oz}\he(\baf,\bag)$,
such that the following diagram commutes:
\[
\xymatrix@C=60pt{
\phi ^{-1}\ff\ar[r]^-{\phi ^{-1}u}\ar[d]^-{\sim}&\phi ^{-1}\g\ar[d]^-{\sim}\\
\pru\baf\oti_{\pru\oz}\he\ore{Z\tim V}\ar[r]^-{v\oti\id}&\pru\bag\oti_{\pru\oz}\he\ore{Z\tim V}
}
\]
\end{definition}
\begin{remark}\label{RemNeufAppendCh4}
Let $\bigl( \xg,\ore{\xg}\bigr)$ be a \rsas. Lemma \ref{LemUnAppenCh4} (i) implies that the category $\coh ^{\rel}(\xg)$ of \rc\ \sv\ on $\xg$ with strict morphisms is an abelian subcategory of $\textrm{Mod}\bigl( \ore{\xg}\bigr)$. 
\begin{definition}\label{DefUnBisAppendCh4}
Let $\bigl( \xg,\ore{\xg}\bigr)$ be a \rsas.
\begin{enumerate}
\item[(i)] We define the \emph{relative analytic $K$-theory of $\xg$} by 
\[
\kre{}(\xg)=\lim_{\genfrac{}{}{0pt}{2}{\longleftarrow}{\xgp}}K\bigl( \textrm{Coh}^{\re}\be(\xgp)\bigr),
\]
where $\xgp$ runs through relatively compact open analytic subsets in $\xg$. 
\item[(ii)] In the same way, if $\zg$ is a relative sub-analytic space of $\xg$, the relative analytic $K$-theory of $\xg$ with suppoort in $\zg$ is defined as
\[
\kre{\zg}(\xg)=\lim_{\genfrac{}{}{0pt}{2}{\longleftarrow}{\xgp}}K\bigl( \textrm{Coh}^{\re}_{\zg\cap\xgp}(\xgp)\bigr),
\]
where $\textrm{Coh}^{\re}_{\zg\cap\xgp}(\xgp)$ is the abelian category of relatively coherent sheaves on $\xgp$ supported in $\zg$ and $\xgp$ runs through relatively compact open analytic subsets in $\xg$.
\end{enumerate}

\end{definition}

\end{remark}

As for coherent \sv, we can define usual operations on relative analytic $K$-theory. 
Here is a list of these operations:
\begin{enumerate}
 \item [(i)] \emph{The product.} A product from $\kre{}(\xg)\oti_{\Z}\he\kre{}(\xg)$\ \  to\ \ $\kre{}(\xg)$ is defined by 
\[
\ff\, .\, \g=\sum_{k\ge 0}(-1)^{k}\tor^{k}(\ff,\g).
\]
Similarly, if $\zg$ is a relative analytic subspace of $\xg$, a product with support from $\kre{}(\xg)\oti_{\Z}\he\kre{\zg}(\xg)$\ \  to\ \ $\kre{\zg}(\xg)$ is given by the same formula.
\item[(ii)] \emph{The dual morphism.} It is an involution of $\kre{}(\xg)$ given by \[
\ff^{\vee}\be=\sum_{k\ge 0}(-1)^{k}\ext^{k}\bigl( \ff,\ore{\xg}\bigr).
\]
\item[(iii)] \emph{The pull-back.} If 
$\apl{f}{\bigl( \xg,\ore{\xg}\bigr)}{\bigl( \xgp,\ore{\xgp}\bigr)}$ is a weak morphism, the pull-back map $\apl{f\pe\be}{\kre{}(\xgp)}{\kre{}(\xg)}$ is defined by 
\[
f\pe\g=\sum_{k\ge 0}(-1)^{k}\be\tor^{k}(\g,f).
\]
If $\zgp$ is a relative analytic subspace of $\xgp$, we also have a pull-back morphism with supports $\apl{f\pe\be}{\kre{\zgp}(\xgp)}{\kre{f^{-1}(\zgp)}(\xgp)}$ given by the same formula.
\item[(iv)] \emph{The Gysin map.} 
If $\apl{f}{\bigl( \xg,\ore{\xg}\bigr)}{\bigl( \xgp,\ore{\xgp}\bigr)}$ is a morphism and $\zg$ is a relative analytic subspace of $\xg$ such that for every $b$ in $B$, $f_{b}\he{}_{\vert \zg_{b}\he}\he $ is finite, the Gysin morphism 
$\apl{f_{*}\he}{\kre{\zg}(\xg)}{\kre{}(\xgp)}$ is induced by the exact functor
$\apl{f_{*}\he}{\textrm{Coh}^{\re}_{\zg}(\xg)}{\textrm{Coh}(\xgp).}$
\end{enumerate}
\par 
We now list all the properties we need relating the operations introduced above.
\begin{proposition}\label{PropUnBisAppendCh4}
 $\he$
\begin{enumerate}
 \item [(i)] If $\bigl( \xg,\ore{\xg}\bigr)$ is a \rsas, $\kre{}(\xg)$ is a unitary ring. Furthermore, if $\zg$ is a relative analytic subspace of $\xg$, $\kre{\zg}(\xg)$ is a module over $\kre{}(\xg)$.
\item[(ii)] The pull-back morphism in relative $K$-theory is contravariant and the Gysin map is covariant.
\item[(iii)] The projection formula holds. More precisely, if $\apl{f}{ \xg}{ \xgp}$ is a morphism,
$\zg$ a relative analytic subspace of $\xg$ such that 
$f_{\vert \zg}\he$ is finite, $\ff$ a \rc\ \sh\ on $\xg$ supported in $\zg$ and $\g$ a \rc\ \sh\  on $\xgp$, then $f_{*}\he\bigl( \ff\,.\,f\pe\be\g\bigr)=f_{*}\he\ff\,.\,\g$.
\item[(iv)] Let $\xg$ be a relative smooth analytic space over $B$, $\xgp$ and $\Delta $ be two \rsas s over $B'$ and $\apl{f}{ \xg}{ \xgp}$ a weak morphism; $f$ induces a weak morphism
$\apl{f_{\Delta }\he}{\xg\tim_{B}\he\bigl( \Delta\tim _{B'}\he B\bigr)}{\xgp\tim_{B}\he\Delta.}$
Let $\zgp$ be a relative analytic subspace of $\xgp\tim_{B}\he\Delta $ such that the projection $\apl{q}{\xgp\tim_{B}\Delta }{\xgp}$ is finite on $\zgp$, and let $\zg=f^{-1}_{\Delta }(\zgp)$. We consider the diagram
\[
\xymatrix@C=40pt{
\xg\tim_{B}\he\bigl( \Delta\tim _{B'}\he B\bigr) \ar[r]^-{f_{\Delta }\he}\ar[d]_-{p}&\xgp\tim_{B}\he\Delta \ar[d]^-{q}\\
\xg\ar[r]_-{f}&\xgp
}
\]
as well as the following pull-back and push-forward operations:
\begin{align*}
 \hspace*{12mm}\apl{f\pe\be&}{\kre{}(\xgp)}{\kre{}(\xg)}&\apl{\!\!f_{\Delta }\pe&}{\kre{\zgp}(\xgp\tim_{B}\he\Delta )}{\kre{\zg}\bigl[ \xg\!\tim_{B}\he\bigl( \Delta\tim _{B'}\he B\bigr)\bigr]}\\
\hspace*{12mm}
\apl{p_{*}\he&}{\kre{\zg}\bigl[ \xg\tim_{B}\he\bigl( \Delta\tim _{B'}\he B\bigr)\bigr] }{\kre{}(\xg)}&\apl{q_{*}\he&}{\kre{\zgp}(\xgp\tim_{B}\he\Delta )}{\kre{}(\xgp)}
\end{align*}
Then, for every $\g$ in $\coh_{\zgp}\he(\xgp\tim_{B}\he\Delta )$
we have 
$
f\pe\be q_{*}\he\g=p_{*}f\pe_{\Delta }\g$ in $ \kre{}(\xg)$.
\item[(v)] Let $\xg$, $\Delta $ be two relative smooth analytic spaces over $B$, $\xgp$ a relative smooth analytic subspace of $\Delta $ and $\apl{f}{\xg}{\xgp}$ a morphism. Consider the cartesian diagram
\[
\xymatrix@C=40pt{
\xg \ar[r]^-{(\id,i\circ f)}\ar[d]_-{f}&\xg\tim_{B}\he\Delta \ar[d]^-{f_{\Delta }\he}\\
\xgp\ar[r]_-{(\id,i)}&\xgp\tim_{B}\he\Delta
}
\]
as well as the following pull-back and push-forward operations:
\begin{align*}
 \hspace{2mm}\apl{&f\pe\be}{\kre{}(\xgp)}{\kre{}(\xg)}&&\apl{f_{\Delta }\pe}{\kre{\xgp}(\xgp\tim_{B}\he\Delta )}{\kre{\xg}(\xg\tim_{B}\he\Delta )}\\
\hspace{2mm}\apl{&(\id,i)_{*}\he}{\kre{}(\xgp)}{\kre{\xgp}(\xgp\tim _{B}\he\Delta )}&&\apl{(\id,i\circ f)_{*}\he}{\kre{}(\xg)}{\kre{\xg}(\xg\tim_{B}\he\Delta )}
\end{align*} 
Then for every \rc\ \sh\ $\g$ on $\xgp$, we have $f_{\Delta }\pe(\id,i)_{*}\he\g=(\id,i\circ f)_{*}\he f\pe\be\g$ in $\kre{\xg}(\xg\tim_{B}\he\Delta )$.
%
\end{enumerate}

\end{proposition}
\begin{proof}
 (i) If $\ff$, $\g$ and $\hh$ are relatively coherent sheaves on $\xg$, for each $\xgp$ open and relatively compact in $\xg$, we have a spectral sequence such that 
\[
\begin{cases}
 E^{\,p,\,q}_{2}&\!\!\!\!=\tor^{\,p}_{\ore{\xgp}}\bigl( \tor^{q}_{\ore{\xgp}}(\ff,\g),\hh\bigr)\\
E^{\,p,\,q}_{\infty }&\!\!\!\!=\gr_{p}\,\he\tor^{\,p+q}_{\ore{\xgp}}(\ff,\g,\hh).
\end{cases}
\]
and $E^{\,p,\,q}_{2}=0$ except for finitely many couples $(p,q)$. Furthermore, by Lemma \ref{LemUnAppenCh4} (i), the \sv\ $E_{2}^{\,p,\,q}$ are relatively coherent on $\xgp$ and the morphisms $d_{2}^{\,p,\,q}$ are strict. Therefore, for all $r\ge 2$, the \sv\ $E_{r}^{\,p,\,q}$ are relatively coherent and the morphisms $d_{2}^{\,p,\,q}$ are strict so that 
\[
\sum_{p,\,q\ge 0}(-1)^{p+q}E^{\,p,\,q}_{2}=\sum_{n\ge 0}(-1)^{n}\,\tor^{\,n}_{\ore{\xgp}}(\ff,\g,\hh)
\] in $\kre{}(\xgp)$. This gives the result. 
\par\medskip 
(ii) The proof is similar to the proof of (i), using spectral sequences associated to the composition of two functors.
\par\bigskip 
\noindent The proofs of (iii), (iv) and (v) are done in the same way. We will prove (iv).
\par\medskip 
(iv)  Let $x\in\xg$. We take charts $U_{x}\he\simeq Z\tim V$ and $U_{f(x)}\he\simeq Z'\tim V'$ such that $\apl{f}{U_{x}\he}{U_{f(x)}}$ has the form $\flgdba{(z,v)}{(g(z),u(v))}$, where $\apl{g}{Z}{Z'}$ is holomorphic and $\apl{u}{V}{V'}$ is smooth. Let $\delta _{_{1}}\he,\dots,\delta _{_{N}}\he\in \Delta $ such that 
$q^{-1}(f(x))\cap\zg=\cup_{i=1}^{N}\bigl( f(x),\delta _{_{i}}\he\bigr)$. We take a chart 
$U_{\delta _{_{1}},\dots,\,\delta _{_{N}}}\he\simeq Y\tim V'$ in a \nbh\ of the $\delta _{_{i}}\he$'s. Then the diagram looks locally on $\xgp$ as follows:
\[
\xymatrix@C=40pt@R=30pt{
Z\tim Y\tim V\ar[r]^-{(g,\,\id,\,u)}\ar[d]_-{\pr_{13}\he}&Z'\tim Y\tim V'\ar[d]^(.55){\pr_{13}}\\
Z\tim V\ar[r]_-{(g,\,u)}&Z'\tim V'}
\]
Furthermore, $\zg=\ba{\zg}\tim V$ where $\ba{\zg}$ is an analytic subset of $Z\tim Y$ and $\pr_{_{1}}\he{}_{\!\!\vert\ba{\zg}}\he$ is finite. Then for any analytic \sh\ $\ba{\g}$ on $Z'\tim Y$, we have
\[
(g,u)\ee\be\pr_{_{13}*}\he\bigl( \pr_{_{12}}^{-1}\ba{\g}\oti_{\pr_{12}^{-1}\oo_{Z'\tim Y}\he}\he\!\!\ore{Z'\tim Y\tim V'}\bigr)\simeq
\pr_{_{13}*}\he\,(g,\id,u)\ee\be\bigl( \pr_{_{12}}^{-1}\ba{\g}\oti_{pr_{12}^{-1}\oo_{Z'\tim Y}\he}\he\!\!\ore{Z'\tim Y\tim V'}\bigr).
\]
Taking the derivative with respect to $\ba{\g}$ and using Lemma \ref{LemUnAppenCh4} (i), we obtain the result.
\end{proof}
\subsection{Topological classes}
In Sections \ref{q} and \ref{p}, we have constructed a theory for relative coherent sheaves as well as associated operations. It remains to obtain cohomological informations about these objects. To do so, we will construct global smooth resolutions for relatively coherent \sv. We start with a general result about smooth resolutions.
\begin{proposition}\label{NouvProp}
Let $M$ be a smooth manifold, $G$ a finite group of $\emph{Diff}(M)$ and $Y=M/_{\ds G}$. Let $\hh$ be a sheaf of $\ci_{Y}$-modules which admits a finite free resolution in a \nbh\ of any point $y\in Y$. Then 
\begin{enumerate}
 \item [(i)] $\hh$ admits a finite locally free resolution in a \nbh\ of any compact set of $Y$.
\item [(ii)] Two resolutions of $\hh$ in a \nbh\ of a compact set are sub-resolutions of a third one.
\end{enumerate}
\end{proposition}
\begin{proof}
We will use several times the following lemma:
\begin{lemma}\label{Utile}
 Let 
$Y=M/G$ and let $\sutrgd{\ff}{\g}{\hh}$ be an exact sequence of $\ci_{Y}$-modules on an open set $U\suq Y$ such that $\hh$ is locally free. Then this exact sequence globally splits on $U$.
\end{lemma}
\begin{proof}
 It is sufficient to prove that $\textrm{Ext}^{1}_{\ci_{Y}}(U,\hh,\ff)=0$. The $\eee xt\Longrightarrow \textrm{Ext}$ spectral sequence satisfies
\[
\begin{cases}
 E_{2}^{p,q}=H^{p}\be\Bigl[ U,\eee xt^{q}_{\ci_{Y}}\bigl( \hh,\ff\bigr)\Bigr]\\
E_{\infty }^{p,q}=\textrm{Gr}^{p}\be\, \, \textrm{Ext}^{p+q}_{\ci_{Y}}(U,\hh,\ff).
\end{cases}
\]
Since $\hh$ is locally free, $\eee xt^{q}(\hh,\ff)=0$ for $q>0$ and since $\hh om_{\ci_{Y}}\he(\hh,\ff) $ is a fine sheaf, $E^{p,0}_{2}=0$ for $p\ge 1$. Therefore all the terms $E_{2}^{p,q}$ vanish except $E_{2}^{0,0}$. This implies $\emph{Ext}^{\, 1}_{\ci_{Y}}(U,\hh,\ff)=0$. 
\end{proof}
(i) Let $K\suq Y$ be a compact. We choose a finite covering $\bigl( U_{i}\he\bigr)_{1\le i\le d}\he$ of $K$ and open sets $\bigl( V_{i}\he\bigr)_{1\le i\le d}\he$ such that $\ba{U\he}_{i}\he\suq V_{i}\he$ and $\hh_{\, \vert V_{i}\he} $ admits a finite free resolution. Using smooth cut-off functions, we obtain for each $i$ a complex of sheaves
\[
\xymatrix{0\ar[r]&\bigl( \ci_{Y}\bigr)^{n_{iN}\he}\be\ar[r]&\cdots\cdots\ar[r]&\bigl( \ci_{Y}\bigr)^{n_{i1}\he}\be\ar[r]^(.6){\pi _{i}\he}&\hh\ar[r]&0}
\]
which is exact in $U_{i}\he$. If $E=\bop_{i=1}^{d}\bigl( \ci_{Y}\bigr)^{n_{i1}\he}\be$ and 
$\apl{\pi :=\bop_{i=1}^{d}\pi _{i}\he}{E}{\hh}$, the morphism $\pi $ is surjective. Let $\nn_{i}\he=\ker \pi _{i}\he$ and $\nn=\ker \pi $. We have an exact sequence: 
\[
\xymatrix{0\ar[r]&\nn_{i}\he\ar[r]&\nn_{\vert\,  U_{i}\he}\ar[r]&\bop_{j\not=i}\bigl( \ci_{Y}\bigr)^{n_{j1}\he}\be\ar[r]&0.}
\] 
Thus $\nn_{\vert\,  U_{i}\he}$ is locally isomorphic to $ \nn_{i}\he\oplus\bigl( \ci_{Y}\bigr)^{\sum_{j\not = i}n_{j1}\he}\be$. Furthermore $\nn_{i}\he$ admits a finite free resolution of length $N-1$. Thus $\nn$ admits a finite free resolution of length at most $N-1$ in a \nbh\ of every point in $K$ and we can start the argument again. After at most $N$ steps, the kernel will be locally free.
\par\medskip 
(ii) Let $\bigl( E_{i}\he\bigr)_{1\le i\le N}\he$ and $\bigl( F_{i}\he\bigr)_{1\le i\le N}\he$ be two finite locally free resolutions in a \nbh\ of $K$. Suppose that we have constructed $\bigl( G_{i}\he\bigr)_{1\le i\le k}\he$ and injections
$\xymatrix{
E_{_{\bullet}}\he\ar@{^{(}->}[r]& G_{_{\bullet}}\he
}$ and $\xymatrix{
F_{_{\bullet}}\ar@{^{(}->}[r]& G_{_{\bullet}}\he
}$
Let $Q_{k}\he=G_{k}\he/E_{k}\he$ and $R_{k}\he=G_{k}\he/F_{k}\he$.
The sheaves
$Q_{1}\he,\dots,Q_{k}\he,R_{1}\he,\dots,R_{k}\he$ are locally free. Let $N_{k}\he=\ker\bigl(\!\! \xymatrix@C=12pt{E_{k}\he\ar[r]&E_{k-1}\he}\!\!\bigr)$, $N'_{k}=\ker\bigl(\!\! \xymatrix@C=12pt{F_{k}\he\ar[r]&F_{k-1}\he}\!\!\bigr)$,
$N''_{k}=\ker\bigl(\!\! \xymatrix@C=12pt{G_{k}\he\ar[r]&G_{k-1}\he}\!\!\bigr)$,
$\ti{Q}_{k}\he=\ker\bigl(\!\! \xymatrix@C=12pt{Q_{k}\he\ar[r]&Q_{k-1}\he}\!\!\bigr)$ and
$\ti{R}_{k}\he=\ker\bigl(\!\! \xymatrix@C=12pt{R_{k}\he\ar[r]&R_{k-1}\he}\!\!\bigr)$. 
Then
$\ti{Q}_{k}\he$ and $\ti{R}_{k}\he$ are locally free. We have two exact sequences
$\sutrgd{N_{k}\he}{N''_{k}}{\ti{Q}_{k}\he}$ and $\sutrgdpt{N'_{k}}{N''_{k}}{\ti{R}_{k}\he}{.}$ By Lemma \ref{Utile}, $N''_{k}\simeq N_{k}\he\oplus\ti{Q}_{k}\he\simeq N'_{k}\oplus \ti{R}_{k}\he$, and we define $G_{k+1}\he=\bigl( E_{k+1}\he\oplus\ti{Q}_{k}\he\bigr)\oplus\bigl( F_{k+1}\he\oplus\ti{R}_{k}\he\bigr)$. 
We put $G_{N+1}\he=N''_{N}$ to end the resolution $G_{_{\bullet}}\he$.
\end{proof}
We apply now this result in our context.
Let $\ff$ be a relatively coherent \sh\ on a \rsas\ $\xg$ over $B$, where $B=M/G$.
Then $\xg=\bigl( \xg\tim_{B}\he M\bigr)/G$ and $\xg\tim_{B}\he M$ is smooth. Furthermore, by Lemma \ref{LemUnAppenCh4} (ii), the \sh\ $\ff^{\infty }\be:=\ff\oti_{\ore{\xg}}\he\ci_{\xg}$ locally admits finite \mbox{$\ci_{\xg}$-free} resolutions. By Proposition \ref{NouvProp} (i), $\ff^{\infty }\be$ admits a globally locally \mbox{$\ci_{\xg}$-free} resolution $E_{_{\bullet}}\he$ on any relatively compact open analytic subset $\xgp$ of $\xg$. Besides, Proposition \ref{NouvProp} (ii) implies that the element $\sum_{i=1}^{N}(-1)^{i-1}\bigl[ E_{i}\he\bigr]$ in $K(\xgp)$ is independent of the chosen resolution $E_{_{\bullet}}\he$.
\par\medskip 
In conclusion, we can associate to each relatively coherent sheaf $\ff$ on $\xg$ a topological class $\bigl[ \ff^{\infty }\be\bigr]$ in 
$\lim\limits_{\genfrac{}{}{0pt}{2}{\longleftarrow}{\xgp}}K(\xgp)$, where $\xgp$ runs through all open relatively compact analytic subspaces of $\xg$.
We now state properties of this topological class:
\begin{proposition}\label{Compl}
Let $\xg$ be a \rsas\ over $B$.
\begin{enumerate}
 \item [(i)] The  topological class map from $\coh^{\re}\be(\xg)$ to $\lim\limits_{\genfrac{}{}{0pt}{2}{\longleftarrow}{\xgp}}K(\xgp)$ factors through $\kre{}(\xg)$.
\item[(ii)] Let $\ff$ be a relatively coherent \sh\ on $\xg\tim_{B}\he (B\tim[0,1])$ and for all $t\in[0,1]$, let $\apl{i_{t}\he}{\xg\tim_{B}\he (B\tim\{t\})}{\xg\tim_{B\tim\{t\}}\he (B\tim[0,1])}$ be the associated base change morphism. Then the class $\bigl[ i_{t}\ee\ff^{\infty }\be\bigr]$ in $\lim\limits_{\genfrac{}{}{0pt}{2}{\longleftarrow}{\xgp}}K(\xgp)$ is independent of $t$.
\end{enumerate}
\end{proposition}
\begin{proof}
(i) We must show that if $\sutrgd{\ff}{\g}{\hh}$ is a strict exact sequence of rela\-tively coherent sheaves, then $[\ff^{\, \infty }\be]-[\g^{\, \infty }\be]+[\hh^{\, \infty }\be]=0$. This sequence is locally isomorphic to
\[
\sutrgdpt{pr^{-1}_{1}\, \ba{\ff\he}\oti_{pr^{-1}_{1}\oo_{Z}\he}\oo_{Z\tim V}^{\, \rel}}{pr^{-1}_{1}\, \ba{\g\he}\oti_{pr^{-1}_{1}\oo_{Z}\he}\oo_{Z\tim V}^{\, \rel}}{pr^{-1}_{1}\, \ba{\hh\he}\oti_{pr^{-1}_{1}\oo_{Z}\he}\oo_{Z\tim V}^{\, \rel}}{,}
\]
and is obtained by extension of the structure sheaves from an exact sequence of coherent analytic sheaves 
$\sutrgd{\ba{\ff\he}}{\ba{\g\he}}{\ba{\hh\he}}$ on $Z$. We can construct locally free resolutions $E_{\, \ba{\ff\he},\bullet }\he$, $E_{\, \ba{\g\he}, \bullet }\he$, $E_{\, \ba{\hh\he}, \bullet }\he$ of $\ba{\ff\he}$, $\ba{\g\he}$, $\ba{\hh\he}$ related by an exact sequence
$
\sutrgdpt{E_{\, \ba{\ff\he},\bullet }\he}{E_{\, \ba{\g\he}, \bullet }\he}{E_{\, \ba{\hh\he}, \bullet }\he}{.}
$
Using cut-off functions again, we patch these exact sequences together step by step and obtain resolutions 
$E_{\ff^{\, \infty },\bullet }\he$, $E_{\g^{\, \infty },\bullet }\he$, $E_{\hh^{\, \infty },\bullet }\he$ of 
$\ff^{\, \infty }\be$, $\g^{\, \infty }\be$, $\hh^{\, \infty }\be$ on $\xg'$ related by an exact sequence \[
\sutrgdpt{E_{\ff^{\, \infty },\bullet }\he}{E_{\g^{\, \infty },\bullet }\he}{E_{\hh^{\, \infty },\bullet }\he}{.}
\]
\par\medskip 
(ii) Let $\xgp$ be an open relatively compact analytic subset of $\xg$ and $E_{_{\bullet}}\he$ be a resolution of $\ff^{\infty }\be$ on $\xgp$. The class $\alpha (t)=\sum_{i=1}^{N}(-1)^{i-1}\bigl[ i_{t}\ee\,E_{i}\he\bigr]$ in $K(\xgp)$ is independent of $t$.
Since $\ff^{\infty }\be$ is flat over $\ci_{B\tim[0,1]}$, it is also flat over $\ci_{[0,1]}$, so that $\alpha (t)=\bigl[ i_{t}\ee\ff^{\infty }\be\bigr]$ in $K(\xgp)$.
 \end{proof}
\begin{remark}
 If $\zg$ is a relative analytic subspace of $\xg$, there is also a topological class map with support from $\textrm{Coh}_{\zg}^{\re}(\xg)$ to $\lim\limits_{\genfrac{}{}{0pt}{2}{\longleftarrow}{\xgp}}K_{\zg\cap\xgp}\he\,(\xgp)$ which factors through $K_{\zg}^{\re}(\xg)$. 
\end{remark}

\end{document}